\documentclass[a4paper,12pt,openright]{article}
\usepackage{amsfonts,graphicx}
\usepackage{amssymb,amsfonts}
\usepackage{amsmath}
\usepackage{amsthm}
\usepackage[english]{babel}
\numberwithin{equation}{section}
\usepackage{color}
\usepackage{flafter}
\usepackage{pstricks}
\usepackage{graphicx}
\usepackage[rightcaption]{sidecap}
\usepackage{subfig}
\usepackage[all,arc]{xy}
\usepackage{enumerate}
\usepackage{mathrsfs}
\usepackage[pdftex]{hyperref}
\usepackage{indentfirst}
\usepackage{multicol}
\usepackage{float}
\usepackage{ucs}
\usepackage[utf8x]{inputenc}
\usepackage{longtable}
\usepackage{hhline}
\usepackage{pict2e}

\usepackage{makeidx}
\makeindex

\usepackage{geometry}
\usepackage{fancyhdr}

\usepackage{pdfpages}

\theoremstyle{definition}
\newtheorem{theorem}{Theorem}[section]
\newtheorem*{thno}{Theorem}

\newtheorem{lemma}[theorem]{Lemma}
\newtheorem{remark}[theorem]{Remark}
\newtheorem{definition}[theorem]{Definition}

\newtheorem{example}[theorem]{Example}

\begin{document}

\title{Special Weingarten surfaces  with  planar convex boundary}
\author{B. Nelli, G. Pipoli, M.P. Tassi}

\date{}

\maketitle

{\small \noindent {\bf Abstract:} We prove a Ros-Rosenberg theorem in the setting of Special Weingarten surfaces. We show that a compact, connected, embedded,  Special  Weingarten surface in ${\mathbb R}^3$ with  planar convex  boundary is a  topological disk under mild suitable assumptions.
\medskip

 \let\thefootnote\relax\footnote{The authors were partially supported by INdAM-GNSAGA, PRIN-2022AP8HZ9, PRIN-20225J97H5.}

\noindent {\bf MSC 2020 subject classification:}  {53A10, 53C42}

\noindent{{\bf Keywords:}  Weingarten surfaces, ellipticity, lima\c con, topological type.}

\bigskip

\section*{Introduction}

Classically, an oriented surface $\Sigma$ immersed in ${\mathbb R}^3$ is called a {\em Special Weingarten surface}  if its principal curvatures $\kappa_1, \kappa_2$ satisfy a relation 
\begin{equation}\label{defWeingarten1}
W(\kappa_1,\kappa_2)=0,
\end{equation}
where $W: \mathbb{R}^2 \rightarrow \mathbb{R}$ is a function of class $C^1$ and  

\begin{equation}\label{conditionellipticity}
\frac{\partial W}{\partial \kappa_1} \frac{\partial W}{\partial \kappa_2} > 0,
\end{equation}
on the subset of $\mathbb{R}^2$ given by $W^{-1}(\{0\})$. 

Condition \eqref{conditionellipticity} guarantees that the partial differential equation locally satisfied by a Special Weingarten surface, when it is parametrized as a graph over its tangent plane, is elliptic. Notice that, in general, such equation is fully non linear. For short, we will denote  a Special Weingarten surface by SW-surface.

Simple examples of SW-surfaces are constant mean curvature surfaces, that is $W(\kappa_1,\kappa_2) = \kappa_1 + \kappa_2 - 2H$ where the constant $H$ is the mean curvature and surfaces of positive constant Gaussian curvature, that is  $W(\kappa_1,\kappa_2) = \kappa_1 \kappa_2 - K$ where the constant $K$ is the Gaussian curvature.

In the 1950's, A.D. Alexandrov, H. Hopf, S. Chern, P. Hartman \& A. Wintner \cite{Al, Ho, Ch1, Ch2, HaWi} investigated closed Special Weingarten surfaces, extending results holding for constant mean curvature surfaces. At the end of the last century, H. Rosenberg \& R. Sa Earp \cite{RoSa} and F. Brito \& R. Sa Earp \cite{BrSa} studied both the compact with boundary case and the non-compact one, while Sa Earp \& Toubiana \cite{SaTo1} produced examples of SW-surfaces of revolution and R. Bryant \cite{Br} studied closed SW-surfaces using complex analysis. More recently many new results and examples about Special Weingarten surfaces  were produced by different authors \cite{AlEsGa, BuOr, CaCa, CoFeTe, EsMe, FeGaMi, GaMaMi, GaMi1, GaMi2, GaMiTa, KuSt, SaTo1, SaTo2}.

We point out that, in most of the cited works, $W$ is expressed in terms of the mean curvature $H$ and the Gaussian curvature $K$, resulting in the symmetry of $W$ with respect to the principal curvatures. We do not make this assumption and the absence of symmetry leads to the non-differentiability of the underlying partial differential equation at the umbilical points of the SW-surface. Then, as we need a tangency principle, we will rely on  a result by  Alexandrov \cite{Al},  where the symmetry of the elliptic operator is not assumed. We will discuss this issue in  Section \ref{section-maximum-principle}. 

Moreover, many results in the Literature are for Special Linear Weingarten surfaces, that is, surfaces whose mean curvature $H$ and Gaussian curvature $K$ satisfy a relation of the form $2aH + bK = c$, for some $a,b,c \in \mathbb{R}$ with $a^2 + bc > 0$.  It is interesting to notice that  Special Linear Weingarten surfaces are critical points of a natural functional that is a linear combination with constant coefficient depending on $a,$ $b,$ $c,$  of the area, the  integrals of the support function and  of the mean curvature of the surface \cite{AlEsGa, GaMaMi, RoSa}.

For SW-surfaces, the  topological and geometrical structure of compact SW-surfaces whose boundary is a given Jordan curve in ${\mathbb R}^3$ is far from being understood. We shall consider the following question: which are the shape and  the topology of embedded SW-surfaces with convex boundary in a plane $P$, contained in one of the two halfspaces determined by $P$?

We are  going to prove that, under suitable conditions, such  surface is topologically a closed disk. Moreover it is the union of vertical and possibly cylindrical graphs. This is a generalization of the following analogous result for constant mean curvature surfaces in ${\mathbb R}^3$ proved by A. Ros \& H. Rosenberg.

\begin{thno} \cite[Theorem $2$]{RoRo}
Let $\Gamma \subset P$ be a strictly convex curve. There is an $H(\Gamma) > 0$, depending only on the extreme values of the curvature of $\Gamma$, such that whenever $M \subset \mathbb{R}^3_{+}$ is asurface with constant mean curvature $H$ bounded by $\Gamma$, with $0 < H < H(\Gamma)$, then $M$ is topologically a disk and either $M$ is a graph over the domain $\Omega$ bounded by $\Gamma$ or $N = M \cap (\Omega \times ]0, \infty[)$ is a graph over $\Omega$ and $M - N$ is a graph over a subannulus of $\Gamma \times ]0,\infty[$, with respect to the lines normal to $\Gamma \times ]0,\infty[$.
\end{thno}

 Notice that our proof follows the strategy of \cite{Se, NePi, NePiRu}, which do not make use of techniques as curvature estimates and compactness theorems, which are not known in the context of SW-surfaces.

In order to establish Ros-Rosenberg theorem for constant mean curvature surfaces, the authors need to assume  that  the mean curvature remains sufficiently small in comparison to the curvature of the boundary. Due to the broadness of the set of  Special Weingarten surfaces, we need a strategy for identifying the suitable conditions for generalizing Ros-Rosenberg result  to SW-surfaces. Let us outline our plan.

Without loss of generality we can assume that $\kappa_1 \geq \kappa_2$. Then, condition \eqref{conditionellipticity} implies that each connected component of $W^{-1}(\{0\})$ can be parametrized as a graph of the form
\begin{equation} \label{defWeingarten2}
\kappa_2 = g(\kappa_1),
\end{equation}
for some proper function $g:I_g \rightarrow \mathbb{R}$ of class $C^1$ with $g' < 0$ everywhere, where $I_g = [\alpha, b)$, for some $b \in \mathbb{R} \cup \{+\infty\}$ and $g(\alpha) = \alpha$. 

Notice that the fact of expressing Weingarten surfaces using the function $g$ yields implicitly that the corresponding function  $W$ is, in general, not symmetric with respect to its variables. This permits to consider, for example, the very natural cases $\kappa_2 = c \kappa_1$ for any constant $c < 0$, that produce the non-symmetric  function
$W(\kappa_1,\kappa_2) = \kappa_2 - c \kappa_1$.

We denote by $\mathcal{W}_g$ the set of all oriented surfaces immersed in $\mathbb{R}^3$ associated to a given function $g$ satisfying the conditions above.  Then the possible sets $\mathcal{W}_g$ can be divided into three main groups, according to the geometric behaviors of their elements. Surfaces belonging to each group are called of  minimal-type, CMC-type (behaving as constant mean curvature surfaces) and CGC-type (behaving as constant Gaussian curvature surfaces), and they will be precisely described in terms of $g$ in Section \ref{section-curvaturediagram}.

The first step towards a generalization of Ros-Rosenberg Theorem for SW-surfaces is to notice that for any set $\mathcal{W}_g$ of either minimal-type or CGC-type, without any further restriction on the function $g$, the proof is a  straightforward  consequence of the Convex Hull Property and the Gauss-Bonnet Theorem, respectively (see Remark \ref{minimal-CGC-type}). On the other hand, when $\mathcal{W}_g$ consists of SW-surfaces of CMC-type, the situation is not elementary, and we present two settings where applying our approach.

First, we consider those sets $\mathcal{W}_g$ of CMC-type whose corresponding surfaces are related by some homothety of $\mathbb{R}^3$ and we define an equivalence relation among the functions $g$. If a function $g$ satisfies an additional but natural  global analytic assumption, we get Ros-Rosenberg result for all surfaces obtained from surfaces in $\mathcal{W}_g$ after a suitable homothety (see Theorem \ref{main-theorem}). This first case includes, in particular, any set $\mathcal{W}_g$ of uniformly elliptic Weingarten surfaces, that is, when there exist constants $c_0 < C_0 < 0$ such that $c_0 \leq g'(t) \leq C_0$ for all $t \in I_g = [\alpha,+\infty)$.

In the latter setting, we do not place any global assumption on the behavior of the function $g$, but we impose instead a mild geometric condition involving the curvatures of the sphere and the cylinder belonging to $\mathcal{W}_g$ (see Theorem \ref{theorem2}). As a consequence, we  enlarge the possible sets $\mathcal{W}_g$ of CMC-type for which Ros-Rosenberg result holds. For example, we can consider a function $g$ whose  whose  graph  is asymptotic to a vertical line, that  clearly does not satisfy the assumptions of Theorem \ref{main-theorem}.
It is also important to notice  that Theorem \ref{main-theorem} and Theorem \ref{theorem2} are not consequence one of the other, but they are complementary, although their proofs shares most of the  arguments.

The paper is organized as follows. Section \ref{section1} is devoted to preliminaries. In particular, we give the precise definition of SW-surface and we introduce the curvature diagram of a surface. Moreover,  we define an equivalence relation for sets $\mathcal{W}_g$ of CMC-type related by a  homothety of ${\mathbb R}^3$ and we recall  the tangency principle for SW-surfaces. In Section \ref{section-catenoid},  we present a class of Special Weingarten catenoids, governed by a   linear relation of the form $\kappa_2 = m_0 \kappa_1$, that we intend to use  as barriers in the proofs of theorems \ref{main-theorem} and \ref{theorem2}. In Section \ref{limacon-section}, we describe the geometric properties of  a classic curve called lima\c con of Pascal, that will be useful in our proofs. In Section \ref{section-Lemma}, we prove an important auxiliary lemma. Finally, in Section \ref{section-maintheorems}, we present our  two generalizations of Ros-Rosenberg Theorem for SW-surfaces (Theorem \ref{main-theorem} and Theorem \ref{theorem2})  as well  as their proofs.

\section{Preliminaries on Special Weingarten Surfaces}
\label{section1}
\subsection{The curvature diagram and the definition of  SW-surfaces} \label{section-curvaturediagram}
A convenient way to visualize different sets of SW-surfaces is to consider their curvature diagrams. We recall that for any surface $\Sigma$ immersed in $\mathbb{R}^3$ with principal curvatures $\kappa_1 \geq \kappa_2$, the \emph{curvature diagram} of $\Sigma$ is the set given by
\begin{equation}
\mathcal{D}(\Sigma) = \{(\kappa_1(p),\kappa_2(p)); p \in \Sigma\} \subset \mathbb{R}^2,
\end{equation}
which is contained in the closed semi-plane $x \geq y$ of $\mathbb{R}^2$. The points where $\mathcal{D}(\Sigma)$ intersects the diagonal $x = y$ correspond to the umbilical points of the surface. In general, the curvature diagram of an oriented surface immersed in $\mathbb{R}^3$ may have interior points,  but in the case where $\Sigma$ is a SW-surface, $\mathcal{D}(\Sigma)$ is contained in the level set $W^{-1}(\{0\})$, and then it is the regular curve whose tangent lines have always negative slope (see Fig. \ref{figure1}, left). For example, the curvature diagram of a constant mean curvature surface is contained in a half-line orthogonal to the line $\{x=y\}$, while the curvature diagram of a surface with positive constant gaussian curvature is contained in a branch of a hyperbola (see Fig. \ref{figure1}, right). Other relevant discussion about the curvature diagram can be found in Part II - Chapter V of \cite{Ho}.

\begin{figure}[h]
\centering
\includegraphics[scale=.1]{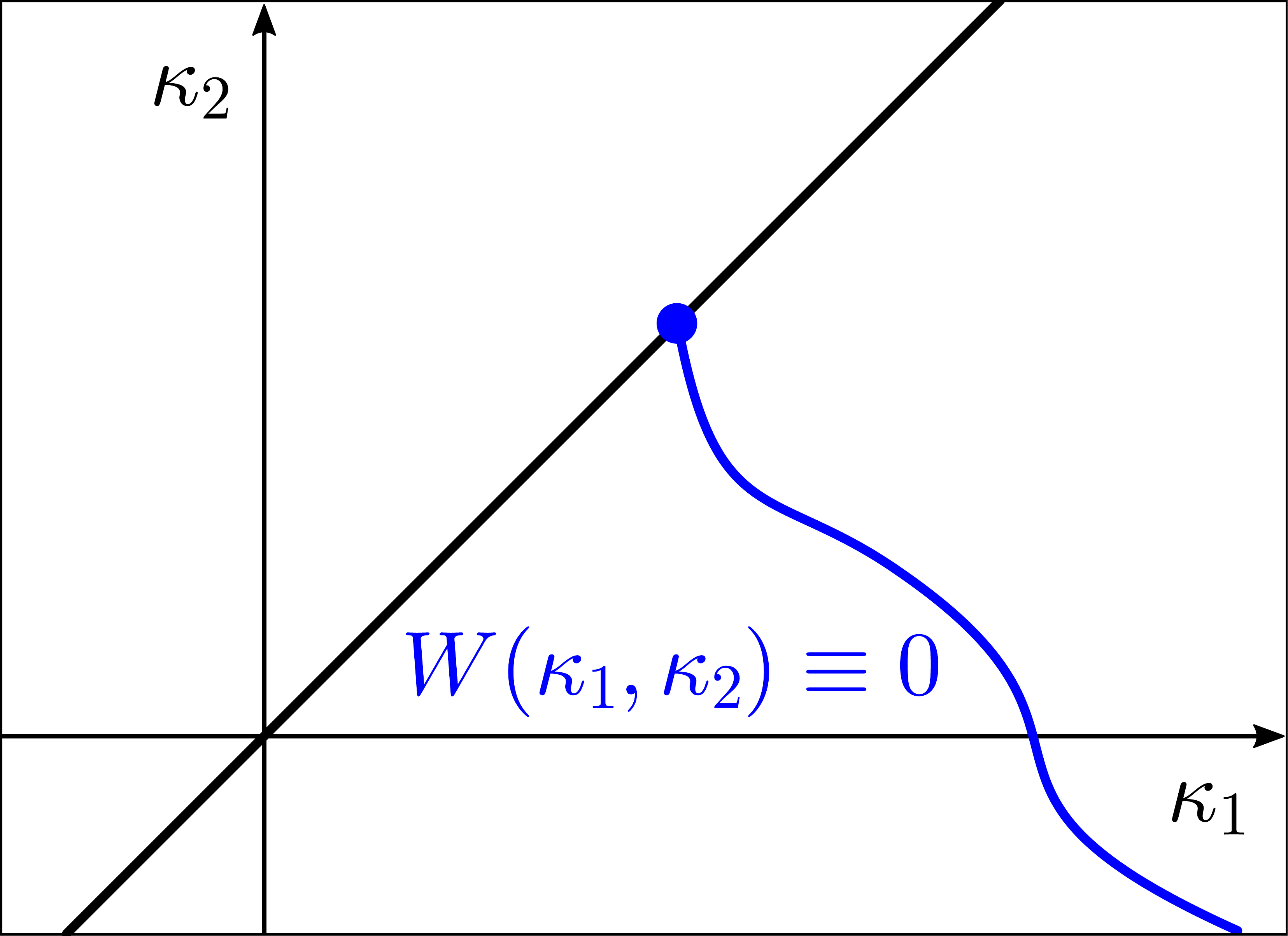}
\hspace{.2cm}
\includegraphics[scale=.1]{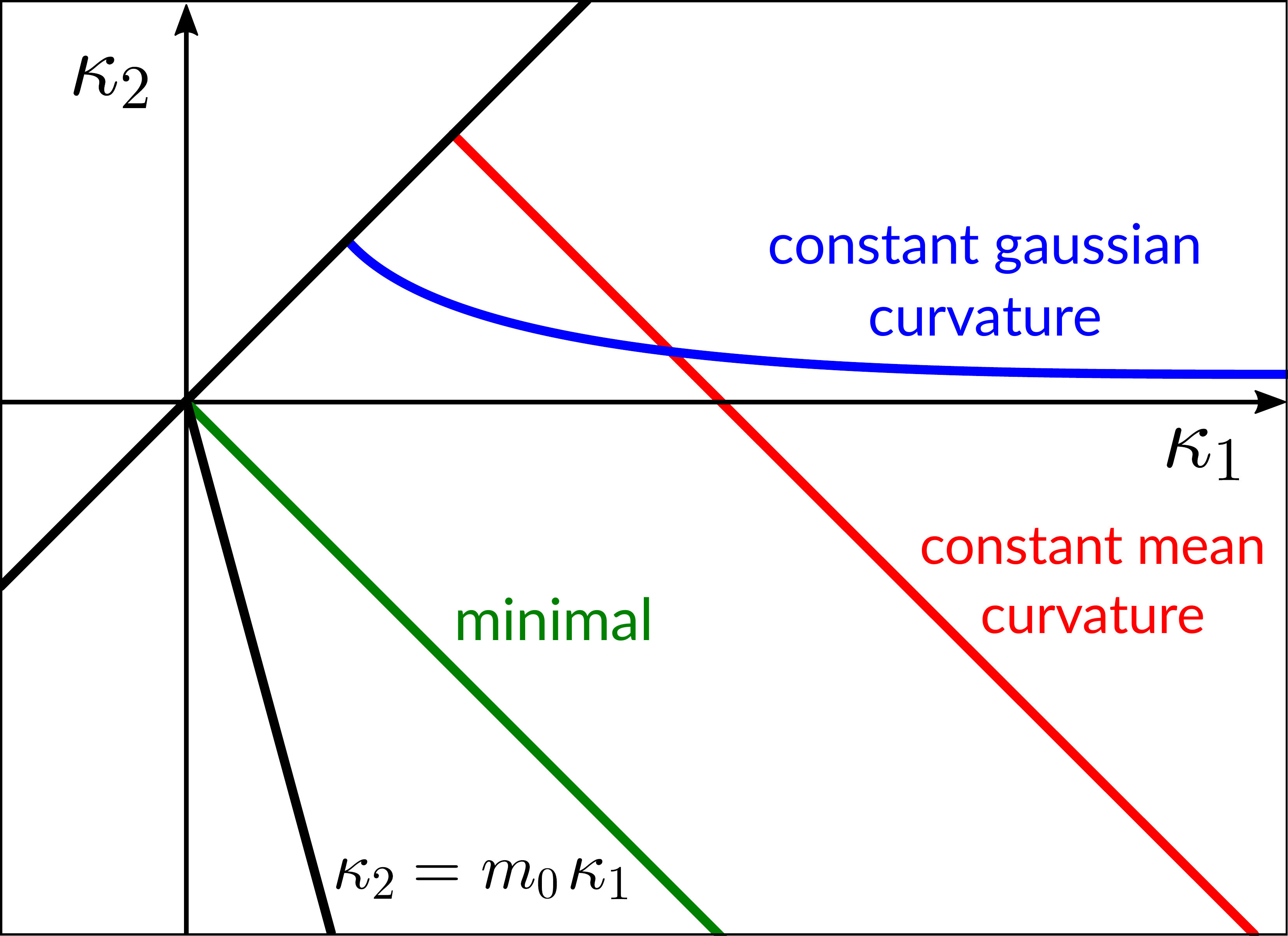}
\caption{on the left, the curvature diagram of a set of SW-surfaces; on the right, curvature diagrams of classical sets of SW-surfaces.}
\label{figure1}
\end{figure}

There are several ways of writing relation \eqref{defWeingarten1} satisfying condition \eqref{conditionellipticity}. In this paper we consider a particular one that is useful for our purposes.

Notice that for a set of SW-surfaces given in terms of Definition \ref{defWeingarten1} satisfying condition \eqref{conditionellipticity}, it is possible that the set $W^{-1}(\{0\})$ may contain more than one connected component. In general, surfaces belonging to different connected component are topologically and geometrically different and they are not comparable by the tangency principle, hence they are not suitable to be considered together. As an example, consider the simple case where $W(\kappa_1,\kappa_2) = \kappa_1 + \kappa_2 - \kappa_1\kappa_2$. Then $W^{-1}(\{0\})$ has two connected components which are contained in different branches of a hyperbola (see Fig. \ref{figure2}, left). When $\kappa_1 > 1$, the branch of this hyperbola includes round spheres of radius $1/2$, while for $\kappa_1 < 1$ the branch of this hyperbola includes planes. 

From now on we always deal with one connected component of $W^{-1}(\{0\})$. Then condition \eqref{conditionellipticity} implies that any such connected component can be parametrized as a graph of the form $\kappa_2 = g(\kappa_1)$ for some $C^1$ function with negative derivative everywhere. This leads to the following redefinition.

\begin{definition}\label{defWeingarten3}
A \emph{Special Weingarten surface} (\emph{SW-surface}, for short) is an oriented surface $\Sigma$ immersed in $\mathbb{R}^3$ whose principal curvatures $\kappa_1 \geq \kappa_2$ satisfy at every point of $\Sigma$ a relation of the type $\kappa_2 = g(\kappa_1)$, for some $g \in C^1(I_g)$ with $g'(t) < 0$ everywhere, where $I_g = [\alpha, b)$ for some $\alpha\in\mathbb R$ and $b \in \mathbb{R} \cup \{+\infty\}$ and $g(\alpha) = \alpha$.
\end{definition}

The number $\alpha$ is called  {\em  umbilicity constant} and it has the following  geometric meaning: when $\alpha = 0,$ planes are solution of \eqref{defWeingarten2}, while for $\alpha \neq 0$, round spheres of radius $1/|\alpha|$ are solutions to \eqref{defWeingarten2}. Up to a  change of orientation, we can always suppose that $\alpha \geq 0$.

We denote by $\mathcal{W}_g$ the set of all oriented SW-surfaces immersed in $\mathbb{R}^3$ associated to a given function $g$.

\begin{figure}[h]
\centering
\includegraphics[scale=.1]{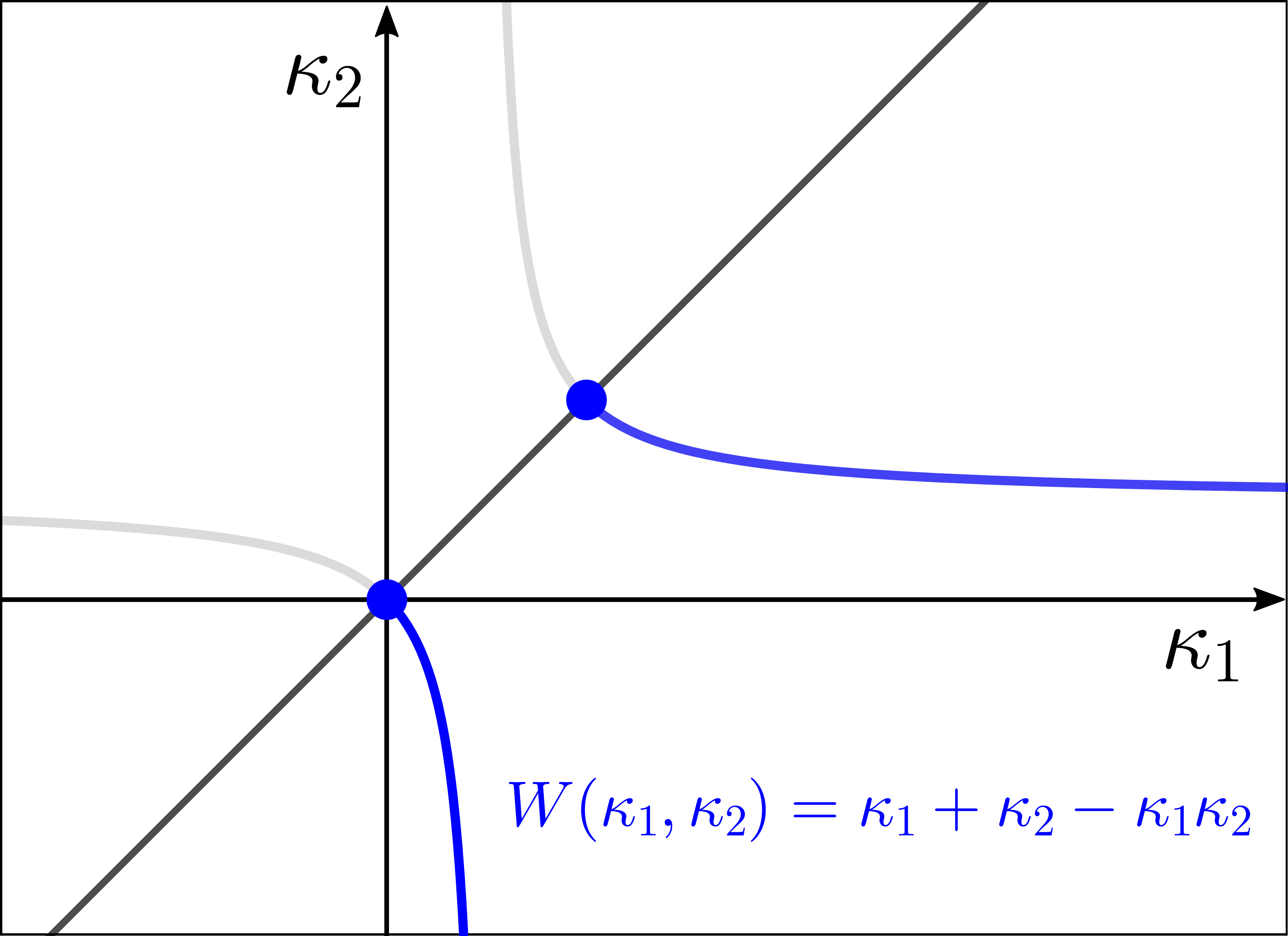}
\hspace{.2cm}
\includegraphics[scale=.1]{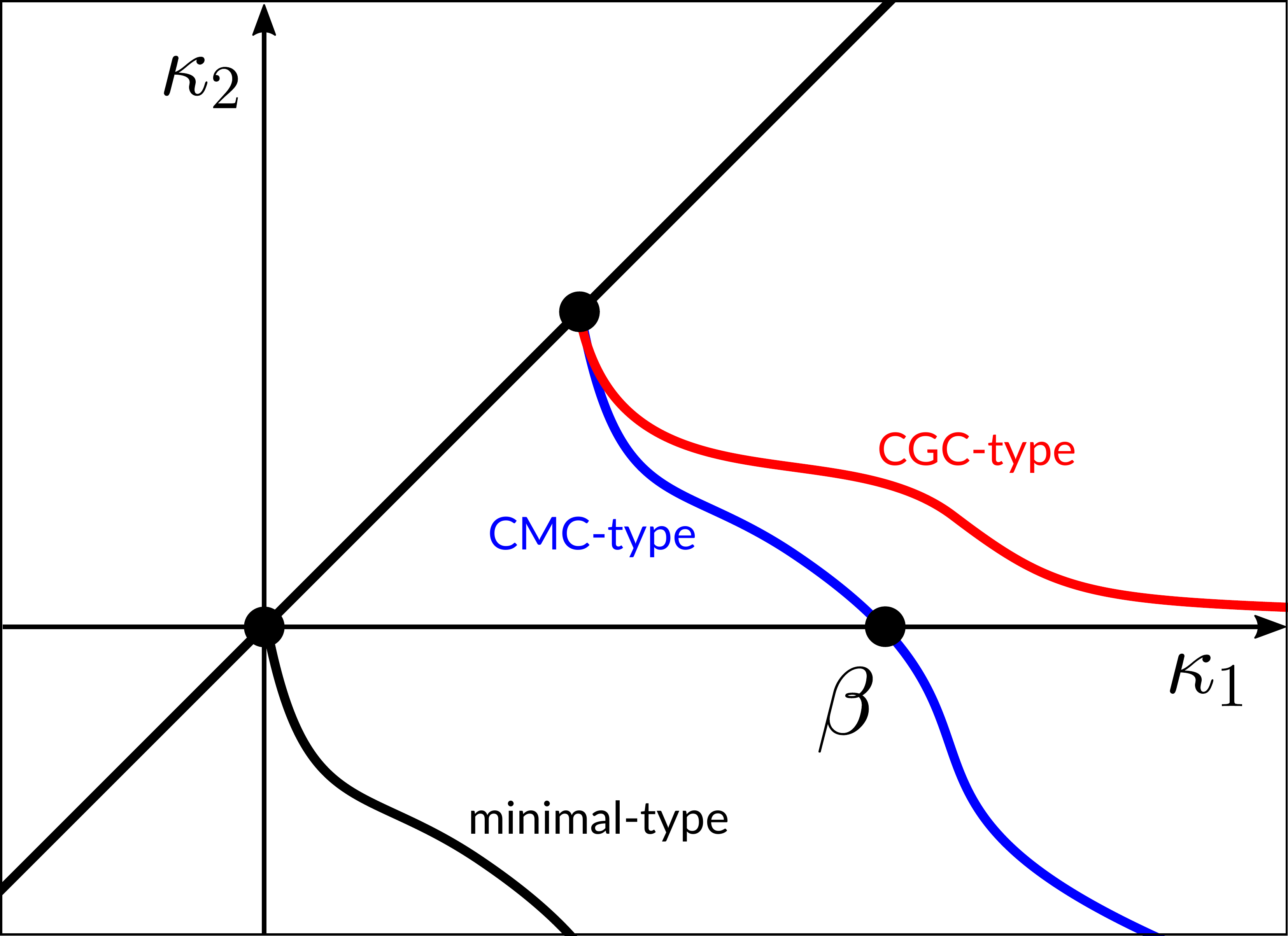}
\caption{on the left, the connected components of $\kappa_1 + \kappa_2 - \kappa_1\kappa_2 = 0$; on the right, curvature diagrams of SW-surfaces of minimal-type, CMC-type and CGC-type.}
\label{figure2}
\end{figure}

We finish this subsection, recalling that, SW-surfaces in $\mathbb{R}^3$ are usually assembled in three main groups, each exemplifying distinct geometric behaviors. We adopt the terminology used in \cite{FeMi} for the  classification of  Weingarten surfaces of revolution. Let $\mathcal{W}_g$ denote a set of elliptic Weingarten surfaces and let $\Sigma$ be a surface in $\mathcal{W}_g$. When $\alpha = 0$ we say that $\Sigma$ is of \emph{minimal-type}; when $\alpha > 0$ and there exists $\beta \in (\alpha,b)$ such that $g(\beta) = 0$, we say that $\Sigma$ is of \emph{CMC-type}; when $\alpha > 0$ and $g(t) > 0$ for all $t \in [\alpha,b)$ we say that $\Sigma$ is of \emph{CGC-type}. Examples of the curvature diagram of each type are pictured in Fig. \ref{figure2}, right. In any set of SW-surfaces of CMC-type case, in particular, we have the existence of round cylinders whose positive principal curvature is the the quantity $\beta$.

\subsection{An equivalence relation among Special Weingarten surfaces} \label{section-equivalencerelation}

In the following we need to assemble those sets $\mathcal{W}_g$ of SW-surfaces whose members are  related by a homothety.
In order to do that, we will establish an equivalence relation in the set of SW-surfaces.

Let  $\Sigma \in \mathcal{W}_g,$ $g:[\alpha, b) \rightarrow \mathbb{R},$ be an immersed surface with principal curvatures $\kappa_1 \geq \kappa_2$. Then, for any $d > 0,$ if $\widetilde{\Sigma}$ denotes the image of $\Sigma$ by a homothety of ratio $d$, its principal curvatures are given by $\widetilde{\kappa}_1 = \kappa_1/d$ and $\widetilde{\kappa}_2 = \kappa_2/d$ and it is easy to check that $\widetilde{\Sigma}$ belongs to  $\mathcal{W}_{\widetilde{g}}$, where $\widetilde{g}:[\alpha/d,b/d) \rightarrow \mathbb{R}$ is given by
\begin{equation*}
\widetilde{g}(t) = \frac{1}{d}g(d t).
\end{equation*}

Now consider  $\mathcal{W}_{g_i}$,  $g_i:[\alpha_i,b_i) \rightarrow \mathbb{R}$, where $\alpha_i \geq 0$ and $b_i \in (\alpha_i,+\infty) \cup \{+\infty\}$, $i=1,2$. We will say that $\mathcal{W}_{g_1}$ and $\mathcal{W}_{g_2}$ are  equivalent if there exists some positive number $d > 0$ such that $\alpha_2 = \alpha_1/d$, $b_2 = b_1/d$ and
\begin{equation} \label{equivalencerelation}
g_2(t) = \frac{1}{d} g_1(d t), \quad \mbox{for all} \quad t \in [\alpha_2,b_2).
\end{equation}
In other words, any surface in $\mathcal{W}_{g_2}$ is obtained as the image of some surface of in $\mathcal{W}_{g_1}$ after a homothety of ratio $d$. The relation defined above is trivially an equivalence relation in the set of all SW-surfaces. It also defines an equivalence relation in the set of functions $\{g:[\alpha,b) \rightarrow \mathbb{R}; \alpha \geq 0, b \in (\alpha,+\infty], g' < 0\}$. We denote by $[g]$ the equivalence class of the function $g$ with respect to the previous equivalence relation. 

For example, all surfaces with positive constant mean curvature can be seen as $\cup_{h\in[h_1]}\mathcal{W}_{h},$ where $h_1: [1,\infty) \rightarrow (-\infty,1]$ is given by $h_1(t) = 2 - t$.

\subsection{The  tangency  principle for Special Weingarten surfaces}
\label{section-maximum-principle}

Due to the ellipticity condition \eqref{conditionellipticity}, the partial differential equation of a SW-surface satisfies a maximum principle, and as a consequence SW-surfaces satisfy a tangency principle, as in the case of constant mean curvature surfaces.
A tangency principle for SW-surfaces can be found in \cite{BrSa} and \cite{Ko}, for example, where authors give a proof based on the classical Hadamard's linearization argument. However, in these works they only consider SW-surfaces in $\mathcal{W}_g$ for which $g'(0)$ $ = -1$. In terms of Definition \ref{defWeingarten1}, this condition is equivalent to require that $W(\kappa_1,\kappa_2)$ is symmetric in its arguments, and for any graph $\Sigma \in \mathcal{W}_g$ the underlying PDE is continuously differentiable with respect to the first and second derivatives.

When the function $W(\kappa_1,\kappa_2)$ is not symmetric in its variables, the underlying PDE is not necessarily continuously differentiable on the umbilical points of the SW-surface, and the Hadamard's linearization argument cannot be applied. A very simple example is given by SW-surfaces governed by the function $W(\kappa_1,\kappa_2) = \kappa_2 - m_0 \kappa_1$, for some $m_0 < 0$. A general tangency principle, including the non symmetric case and allowing to compare SW-surfaces that belong to different set $\mathcal{W}_g$ was proved by A.D. Alexandrov in a series of papers published in the fifties, that are well known but look neglected in some of their parts. We present below an adapted version of Theorem C in \cite{Al}-Part III, which establishes the tangency principle we desire.
\begin{theorem}\label{maximumprinciple}
Let $\mathcal{W}_{g_i}$, $i = 1,2$, be two distinct sets of elliptic Weingarten surfaces immersed in $\mathbb{R}^3$, with respect to the $C^1$ functions $g_i:[\alpha_i,b_i) \rightarrow \mathbb{R}$ satisfying $g'_i(t) < 0$ for all $t \in [\alpha_i,b_i)$, where $b_i \in (\alpha_i,+\infty) \cup \{+\infty\}$. Suppose additionally that $0 \leq \alpha_1 \leq \alpha_2$, $b_1 \leq b_2$ and $g_1(t) \leq g_2(t)$ for all $t \in [\alpha_2,b_2)$.

Let $\Sigma_1 \in \mathcal{W}_{g_1}$ and $\Sigma_2 \in \mathcal{W}_{g_2}$ be two $C^2$ surfaces immersed in $\mathbb{R}^3$ and let $p \in \Sigma_1 \cap \Sigma_2$ be either an interior or boundary tangency point between $\Sigma_1$ and $\Sigma_2$. If $\Sigma_1$ lies above $\Sigma_2$ around $p$, then $\Sigma_1$ coincides with $\Sigma_2$ in a neighborhood of $p$.
\end{theorem}

Roughly speaking, Theorem \ref{maximumprinciple} states that, if the curvature diagram of $\mathcal{W}_{g_1}$ lies in the closed region below the curvature diagram of $\mathcal{W}_{g_2}$, then no surface in $\mathcal{W}_{g_1}$ can have a tangency contact with any surface of $\mathcal{W}_{g_2}$ from above.

\begin{remark}
Let $\Sigma \in \mathcal{W}_g$ be a SW-surface for a given $g$. The reflection of $\Sigma$ with respect to any plane of $\mathbb{R}^3$ is also a surface that belongs to $\mathcal{W}_g$. In other words, $\mathcal{W}_g$ is invariant under reflections of $\mathbb{R}^3$. As a consequence of this fact, the Alexandrov reflection method also holds for any  SW-surface  in $\mathcal{W}_g$ (see chapter VII of \cite{Ho} for a detailed description about this technique).
\end{remark}
 
\section{SW-catenoids}
\label{section-catenoid}
Special Weingarten surfaces with rotational symmetry are described in details and classified in \cite{FeMi}. The minimal-type case was previously treated in \cite{SaTo1} and more recently in \cite{LoPa}. Recall that, if $\Sigma$ is a surface in $\mathcal{W}_g$ for a given $g$ with $\alpha = 0$, then $\Sigma$ is called of \emph{minimal-type}. Regular surfaces invariant by rotations in a set of SW-surfaces $\mathcal{W}_g$ of minimal-type,  when they exist, are called \emph{special catenoids} or \emph{SW-catenoids}.

According to \cite{FeMi, LoPa, SaTo1}, any SW-catenoid is obtained, up to isometries of $\mathbb{R}^3$, as the rotation around the $z$-axis of an unbounded convex planar curve of the plane $\{y=0\}$. Moreover, the corresponding revolution surface is topologically a cylinder and it is symmetric with respect to the horizontal plane $\{z=0\}$. The intersection of the plane $\{z=0\}$ with the SW-catenoid is a circle of minimal radius, called \emph{neck}. We will see that the height of the SW-catenoid with respect to the plane $\{z=0\}$ can be bounded or unbounded. 

Let us consider the particular cases where the set $\mathcal{W}_g$ of elliptic Weingarten surfaces  of minimal-type is given in terms of  $g(t) = m_0 t$, for some $m_0 < 0$, which in view of Theorem 5.3 in \cite{FeMi} admits SW-catenoids. Let $G \in \mathcal{W}_g$ be a SW-catenoid that is inward oriented. After an isometry of $\mathbb{R}^3$ we can assume that the revolution axis is the $z$-axis, and the neck is the horizontal circle $C_0$ of radius $r_0 > 0$ centered at the origin. The part of $G$ above the plane $\{z = 0\}$ can be parametrized as
\begin{equation*}
\psi(s,t) = (s \cos t, s \sin t, h(s)),
\end{equation*}
for $h: [r_0,+\infty) \rightarrow \mathbb{R}$ given by
\begin{equation} \label{heightparametrization}
h(s) = \int_{r_0}^{s}\frac{\tau^{m_0}}{\sqrt{r_0^{2m_0} - \tau^{2m_0}}}d\tau = r_0 \int_{1}^{s/r_0}\frac{1}{\sqrt{w^{-2m_0} - 1}}dw,
\end{equation}

According to \cite{FeMi, LoPa, SaTo1}, its height (w.r.t. the plane $\{z = 0\}$) is bounded if and only if $m_0 < -1$. When $m_0 < -1$ we define
\begin{equation} \label{eqbound}
\mathfrak{h}(m_0) = \lim_{s \rightarrow +\infty} \int_{1}^{s}\frac{1}{\sqrt{w^{-2m_0} - 1}}dw \in (0,+\infty),
\end{equation}
and the quantity $r_0 \mathfrak{h}(m_0)$ express the lowest upper bound to the height attained by $G$. Then, for any value $0 < h^\ast < r_0\mathfrak{h}(m_0)$ we may consider $G^\ast$ as the part of $G$ between the horizontal sections $C_0$ at height $0$ and $C_1$ at height $h^\ast$ (see Fig. \ref{figurecatenoid}, left).
\begin{figure}[h]
\centering
\includegraphics[scale=.81]{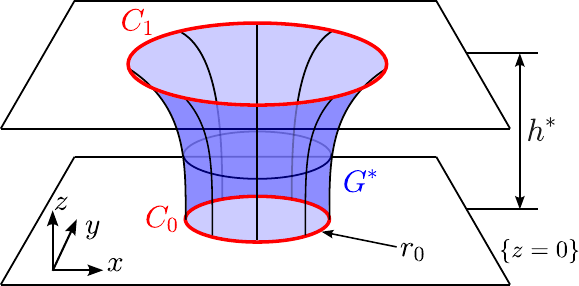}
\hspace{1.5cm}
\includegraphics[scale=.6]{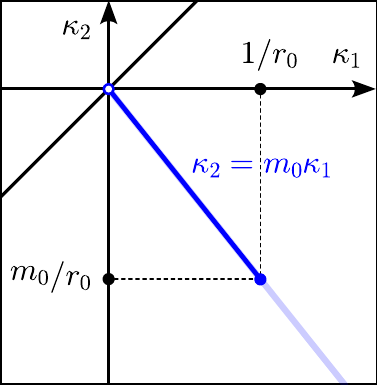}
\caption{on the left, the truncated catenoid $G^\ast$ of height $h^\ast$ and neck-size $r_0$; on the right, the curvature diagram of $G$.}
\label{figurecatenoid}
\end{figure}

Notice that
\begin{align*}
\mathfrak{h}(m_0) &= \lim_{s \rightarrow +\infty} \int_{1}^{s}\frac{1}{\sqrt{w^{-2m_0} - 1}}dw \geq \lim_{s \rightarrow +\infty} \int_{1}^{s}w^{m_0} dw \\
&= \lim_{s \rightarrow +\infty} \frac{s^{m_0 + 1} - 1}{1 + m_0} = \frac{- 1}{1 + m_0}.
\end{align*}
The integral in equation \eqref{eqbound} does not converge for $-1 \leq m_0 < 0$, then, for later use, we define
\begin{equation} \label{h-ast}
\mathfrak{h}^\ast(m_0) =
 \begin{cases}
    \frac{- 1}{1 + m_0} & \mbox{if} \quad m_0 < -1, \\
    1 & \mbox{if} \quad -1 \leq m_0 < 0.
 \end{cases}
\end{equation}

Finally, we note that the curvature diagram of $G$ is always a semi-open line segment joining $(0,0)$ (but not including it) with $(m_0/r_0,1/r_0)$, which corresponds to the principal curvatures of the neck points (see Fig. \ref{figurecatenoid}, right).

\section{The lima\c con of Pascal}
\label{limacon-section}
The \emph{lima\c con of Pascal} is a famous and classical planar algebraic curve. Thought the centuries mathematicians proposed several equivalent constructions of such curve and studied many of its properties. In the present paper the \emph{lima\c con of Pascal} turns out to be an essential tool for estimating  many radii that appear in our proofs. In this section we present the construction that is most convenient for our purposes and we describe the properties that are most useful for our proofs.

\begin{definition}\label{def_lima}
Let $\mathcal C$ be a circle in $\mathbb R^2$ and let $A\in\mathbb R^2$ be a fixed point called \emph{base point}. For any $P\in\mathcal C$ let $A_P$ be the reflection of $A$ through the straight line tangent to $\mathcal C$ in $P$. We call \emph{lima\c con of Pascal} (or just \emph{lima\c con}) the curve 
\begin{equation*}
\mathcal L:=\{\ A_P\in\mathbb R^2\ |\ P\in\mathcal C\ \}.
\end{equation*}
\end{definition}

Denoting by $C$ and $c$ respectively the center and the radius of $\mathcal C$, it is easy to prove that, up to isometries of the Euclidean plane, $\mathcal L$ is uniquely determined  by two positive parameters: $a=\mbox{dist}(A,C)$ and $c$. Moreover $\mathcal L$ depends continuously by $a$ and $c$ (see Fig. \ref{limacon1}, left). 

\begin{figure}[h]
\centering
\includegraphics[scale=.35]{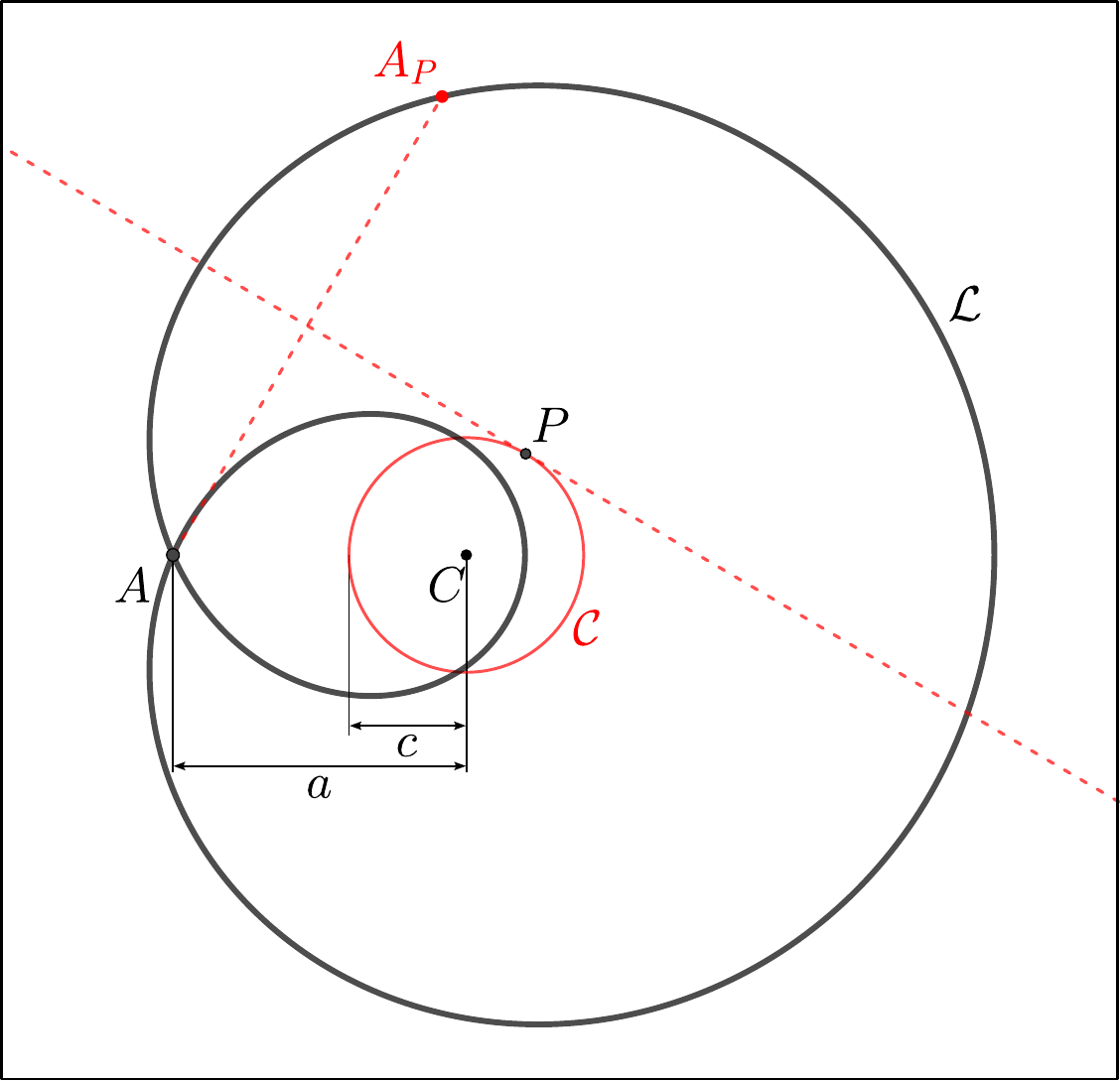}
\hspace{.3cm}
\includegraphics[scale=.35]{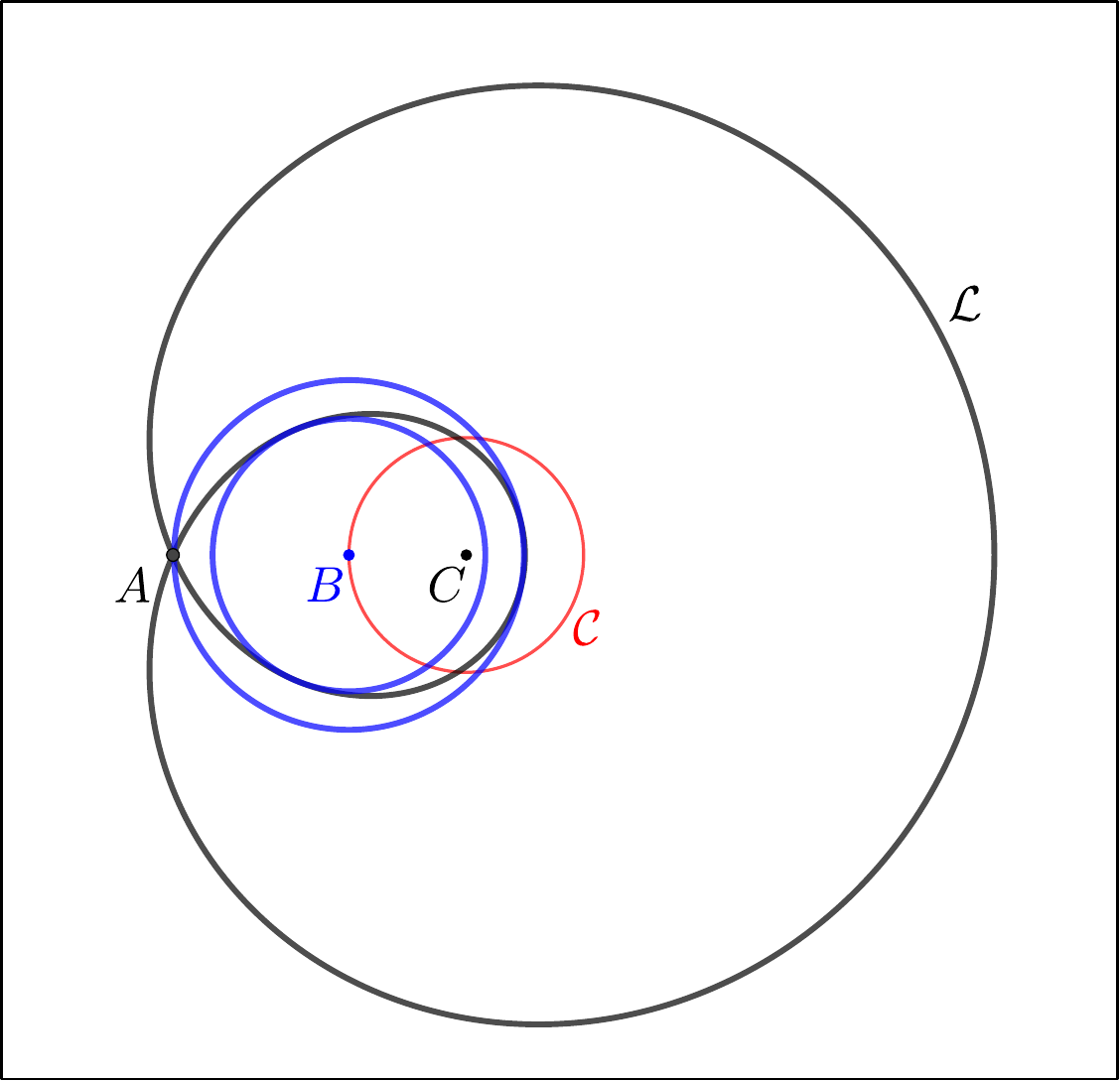}
\caption{on the left, the construction of the lima\c con of Pascal for $a=5$ and $c=2$; on the right, the lima\c con of Pascal and the disks described in Lemma \ref{lemma2_lima}.}
\label{limacon1}
\end{figure}

Let $A=(0,0)$ and $C=(a,0)$. Using Definition \ref{def_lima} and standard arguments, we can find an explicit parametrization of $\mathcal L$ in radial coordinates:
\begin{equation}\label{eq_lima}
\rho(\vartheta)=2a\cos\vartheta+2c,\quad\vartheta\in\mathbb [0,2\pi].
\end{equation}

From \eqref{eq_lima} is easy to deduce the following well known properties of the \emph{lima\c con}.

\begin{lemma}\label{lemma1_lima}
Let $\mathcal L$ be a \emph{lima\c con} with parameters $a,c>0$ and base point $A$:
\begin{enumerate}
\item[(1)] if $a<c$, then $\mathcal L$ is embedded;
\item[(2)] if $a=c$, then $\mathcal L$ has a cusp in $A$;
\item[(3)] if $a>c$, then $A$ is the unique multiple point of $\mathcal L$. Moreover $\mathcal L$ has two loops, one inside the other and branching in $A$.
\end{enumerate}
\end{lemma}

The following result will be very useful for us.

\begin{lemma}\label{lemma2_lima}
Let $\mathcal L$ be a \emph{lima\c con of Pascal} of parameters $a > c > 0$, with respect to a circle $\mathcal C$ and base point $A \in \mathbb{R}^2$. Let $B \in \mathcal{C}$ denote the point of minimal distance to $A$. Then:
\begin{itemize}
\item[(1)] the smaller loop of $\mathcal L$ bounds a disk of radius $\sqrt{\frac{(a-c)^3}{a}}$ centered at $B$;
\item[(2)] the closed  disk of center $B$ and radius $a - c > 0$ contains the smaller loop of $\mathcal L$. 
\end{itemize}
\end{lemma}

\begin{proof}
Using the parametrization \eqref{eq_lima} with base point $A = (0,0)$ and circle $\mathcal{C}$ of radius $c$ centered in $(a,0),$ one has $B = (a-c,0)$.  The lima\c con parametrized by \eqref{eq_lima} is symmetric with respect to $x$-axis and
it has exactly three points on that axis: $A$, one on the smaller loop for $
\vartheta =\pi$ and one on the bigger loop for $\vartheta=0$. Since we have
\begin{equation*}
0 < a-c < 2a-2c = |\rho(\pi)| < 2a+2c = \rho(0),
\end{equation*}
then $B$ belongs to the closed domain bounded by the smaller loop.

We want to compute the minimum and maximum distances from $B$ to the smaller loop of $\mathcal{L}$. Let $l(\vartheta)=(2a\cos\vartheta+2c)(\cos\vartheta,\sin\vartheta)$ be a generic point of $\mathcal{L}$, then we define the function
\begin{equation*}
\varphi(\vartheta)=\mbox{dist}(B,l(\vartheta))^2=4c(a\cos\vartheta+c)(1+\cos\vartheta)+(a-c)^2.
\end{equation*}
After standard computations we can show that 
\begin{equation*}
\min\{\varphi(\vartheta); \vartheta \in [0,2\pi]\} = \frac{(a-c)^3}{a},
\end{equation*}
which proves (1). To prove (2), we need to restrict $\vartheta$ to the interval $[\pi - \arccos(c/a),\pi + \arccos(c/a)]$, which correspond to the points of the smaller loop of $\mathcal{L}$. Again, after standard computations we can show that
\begin{equation*}
\max\{\varphi(\vartheta); \vartheta \in [\pi - \arccos(c/a),\pi + \arccos(c/a)]\} = (a-c)^2,
\end{equation*}
hence (2) follows.
\end{proof}

\begin{remark}\label{rmk_lima}
Note that the function $(a,c)\mapsto \sqrt{\frac{(a-c)^3}{a}}$ is strictly increasing in the variable $a$ and strictly decreasing in the variable $c$.
\end{remark}

\section{The graph Lemma} \label{section-Lemma}

We denote by $D(p,r)$ the horizontal closed disk of radius $r$ and center $p$. Moreover, for any interval $I \subset \mathbb{R}$, we denote by $D(p,r) \times I$ the intersection of the solid cylinder on $D(p,r)$ with the closed slab $\{(x,y,z) \in \mathbb{R}^3; z \in I\}$.}

For convenience, let us recall a well known property about strictly convex planar curves: let $\Gamma \subset \mathbb{R}^2$ be a strictly convex planar curve with maximum and minimum curvature values $\Lambda$ and $\lambda$ satisfying $\Lambda > \lambda > 0$.  Then any circle of radius $1/\lambda$ (resp. $1/\Lambda$) tangent to $\Gamma$, with the same inward normal vector as $\Gamma$ at the tangency point, must bound a closed domain containing $\Gamma$ (resp. contained in the closed domain bounded by $\Gamma$). Moreover, the radius $\omega(\Gamma)$ of the smallest circle in $\mathbb{R}^2$ enclosing $\Gamma$ satisfies $1/\Lambda \leq \omega(\Gamma) \leq 1/\lambda$.

Now, let us prove a result that will be crucial in what follows.

\begin{lemma} \label{Lemasmalldisk}
Consider a set of SW-surfaces $\mathcal{W}_g$ where the umbilicity constant is $\alpha = g(\alpha) > 0$. Assume that $\Sigma \in \mathcal{W}_g$ is a compact surface with boundary such that:
\begin{itemize}
\item[1.] $\Sigma$ is embedded in the closed half-space $\mathbb{R}^3_{+} := \{z \geq 0\}$; 
\item[2.] the boundary of $\Sigma$ is a closed, smooth, strictly convex curve $\Gamma \subset P := \{z = 0\}$;
\end{itemize}
If $\Omega$ is the planar domain bounded by $\Gamma,$ then there exists $p\in \Omega$ depending on $\Gamma$ and $\Sigma$ and $r > 0$ depending only on $\Gamma$ such that $\Sigma \cap (D(p,r)\times [0,+\infty))$ is a vertical graph over $D(p,r)$, with
\begin{equation}
\sqrt{\frac{\lambda}{\Lambda^{3}}}
\leq r(\Gamma)\leq \frac{1}{\Lambda}
\end{equation}
where $\Lambda > \lambda > 0$ are, respectively, the maximum and minimum values of the curvature of $\Gamma$.
\end{lemma}

\begin{proof}
We denote by $\mathcal{O}$ the open region of $\mathbb{R}^3$ bounded by $\Sigma \cup \Omega$.
The strategy of the proof is to find a point  $q$  on  $\Gamma$ such that $\Sigma \cap (\{q\} \times (0,+\infty))$ is only one point. 
Then, the disk $D(p,r)$ of  the statement will be obtained, by reflecting points closed to  $q$  with respect to vertical planes, and using the  Alexandrov reflection method.

In order to find $q\in\Gamma$, we apply the Alexandrov reflection method with horizontal hyperplanes $P(t) := \{z = t\}$ coming down from above. Denote by $h_{\Sigma}$ the height of $\Sigma$ above $P(0)$. The plane $P(t)$,  $t > h_{\Sigma}$, does not intersect $\Sigma$. Then we let $t$ decrease. When $t < h_{\Sigma}$, we reflect the part of $\Sigma$ above $P(t)$ until there is a first contact point between $\Sigma$ and its reflection. If such a contact point does not occur for any $t > 0$, then $\Sigma$ is a vertical graph over $\Omega$ and we can choose $r(\Gamma) = 1/\Lambda$ and $p = q + (1/\Lambda) n(q)$, where $q$ is any point of $\Gamma$ and $n(q)$ is the inward pointing unit normal vector of $\Gamma$ at $q$. 

Otherwise, there will be a $0 < t_0 < h_{\Sigma}/2$ such that the reflected surface touches $\Sigma$ for the first time. If the intersection occurs at an interior point of $\Sigma$ we would have a contradiction with the tangency principle, hence a first touching point must belong to $\Gamma$. Let $q$ be one of such points. Then the line $\{q\} \times (0,+\infty)$ intersects $\Sigma$ exactly once, and $\{q\} \times (0,2t_0)$ is contained in $\mathcal{O}$, as $t_0 < h_M/2$. Note that the portion of $\Sigma$ above $P(t_0)$ is a graph.

Let us now show the existence of the disk $D(p,r)$.

First we notice the following fact.  Let $Q$ be any vertical plane. Since $\Sigma$ is compact, we can assume that $Q \cap \Sigma = \emptyset$. Translate horizontally $Q$ towards $\Sigma$. By abuse of notation, we call again $Q$ any parallel translation of the initial plane. When $Q$ touches $\Sigma$ for the first time, keep moving $Q$ and start reflecting through $Q$ the part of $\Sigma$ left behind $Q$. In order not to have a contradiction with the tangency principle, we can continue this procedure with no contact points between $\Sigma$ and its reflection until $Q$ enters $\Gamma$ at distance at least $1/\Lambda$ from it. In fact let $p'\in \Gamma$ behind  a $Q_v,$ where $Q_v$ is at distance less than $1/\Lambda$ from $p\in\Gamma$.  Then the disk of radius $1/\Lambda$ tangent to $\Gamma$ at $p'$ is contained in $\Omega$ and its center is on the side of $Q_v$ not containing $p$. Hence, the reflection of $p'$ is inside the disk, then it is is inside $\Omega$.
  
Let $\Omega_q$ be the set of such $p'$ obtained with the above construction using just the vertical  planes that leave $q$ behind them. We have that $\Omega_q$ is a subdomain of $\Omega$ with non-empty interior and $\Omega_q\times(0,2t_0)\subset\mathcal O$. Therefore $\Omega_q$ contains the desired disk.  Its shape and size depend on $\Gamma$, but also on $q$. Since we are not able to predict where $q$ is, we can avoid the dependence of the point $q$ stopping the reflection earlier. More precisely we stop when $Q$ becomes tangent to $\mathcal{C}$, where $\mathcal{C}$ is defined as follows. Denote by $\mathcal{C}'$ the horizontal circle of radius $1/\lambda$, tangent to $\Gamma$ at $q$, and enclosing $\Gamma$. Then $\mathcal{C}$ is the circle with the same center as that of $\mathcal{C}'$ and radius equal to $1/\lambda - 1/\Lambda$. 
 
The set of the reflections of $q$ through any vertical plane tangent to $\mathcal{C}$ is a lima\c con of Pascal $\mathcal{L}$ in the plane $P$, with respect to the circle $\mathcal{C}$ and base point $q$. The parameters of $\mathcal{L}$ are $a = 1/\lambda$ and $c = 1/\lambda - 1/\Lambda$, which indicates that the lima\c cons of Pascal obtained starting with  different points  $q \in \Gamma$ are  congruent.  As $a > c$, then $\mathcal{L}$ has two loops. By item (2) of Lemma \ref{lemma2_lima}, the smaller loop $\mathcal{L}'$ of $\mathcal{L}$ is contained in the circle of radius $a - c = 1/\Lambda$ tangent to $\Gamma$ at $q$ hence it is contained in $\Omega_q \cup \{q\}$. 
 
Finally, the conclusion of the statement follows from the fact that the closed domain bounded by the smaller loop of $\mathcal{L}$ contains a closed disk $D(p,r)$, for some $p \in \Omega$ and $r = r(\Gamma)$ greater or equal than $\sqrt{\frac{(a-c)^3}{a}} = \sqrt{\frac{\lambda}{\Lambda^{3}}}$, by item (1) of Lemma \ref{lemma2_lima}, and from similar arguments we used before, $\Sigma$ is a vertical graph over $D(p,r)$.
\end{proof}

\section{Two Ros-Rosenberg type theorems for SW-surfaces} \label{section-maintheorems}

In this section we prove a Ros-Rosenberg type theorem in two cases.

In the former case, we make a natural   assumption on the function $g$ in order to get 
Ros-Rosenberg result for every surfaces  obtained from surfaces in  $\mathcal{W}_g$ after  suitable homotheties, whose ratios can be computed explicitly. We point out that  condition \eqref{assumption1} implies that, when we do a homothety of  a surface in $\mathcal{W}_g,$ its curvature diagram does not intersect  a line with slope $m_0$ passing through the origin. This allows us to compare our surface  with a suitable catenoid  by the tangency principle (see Section \ref{section-maximum-principle}).

In the latter case, we  deal with a broader set of SW-surfaces,  where we do not  place any  assumption on the analytic behaviour of the function $g$. This will prevent us the possibility of doing homotheties, but, as it will be clear later, assuming explicit mild geometric conditions on the umbilicity constant and the curvatures of the standard cylinder belonging to $\mathcal{W}_g$, we are able to prove Ros-Rosenberg result.

\begin{theorem}
\label{main-theorem}
Let $\Gamma \subset P:= \{z = 0\}$ be a strictly convex, closed planar curve and let $g:[\alpha,+\infty) \rightarrow (-\infty,\alpha]$, $\alpha > 0$, be a $C^1$ function satisfying $g(\alpha)=\alpha,$ $g'(t) < 0$ and 
\begin{equation}\label{assumption1}
g(t) \geq (1 - m_0) \alpha + m_0 t, 
\end{equation}
for some constant $m_0 < 0$ and for all $t \in [\alpha,+\infty)$. There exists $d_0 = d_0(\Gamma)$ such that for any $f \in [g]$, where
$f(t)=g(dt)/d$, $t \in [\alpha/d,+\infty)$, for some $d > d_0$, if $\Sigma$ is any compact surface embedded in $\mathbb{R}^3_+ = \{(x,y,z)∈\mathbb{R}^3; z \geq 0\}$ that belongs to  $\mathcal{W}_f$ and whose boundary spans $\Gamma$, then $\Sigma$ is topologically a closed disk. Moreover, if $\Omega$ denotes the domain bounded by $\Gamma$ in $P$ then either $\Sigma$ is a vertical graph over $\Omega$, or $\Sigma \cap(\Omega \times [0,+\infty))$ is a graph over $\Omega$ and $\Sigma - (\Omega \times [0,+\infty))$ is a graph with respect to the lines normal to $\Gamma \times [0,+\infty)$.
\end{theorem}
\begin{remark}
\label{minimal-CGC-type}
We remark that Ros-Rosenberg result  holds for any SW-surface  of either minimal-type or CGC-type with no restriction on the function $g$ and with no need of doing homotheties. The minimal-type case is trivial, since in this case the Gaussian curvature is non-positive and the Convex Hull Property holds (see \cite{Os}). Therefore, if $\Sigma$ is a minimal-type SW-surface satisfying the conditions of  Theorem \ref{main-theorem}, then $\Sigma$ must be the compact planar domain bounded by $\Gamma$ in $P$, which is, in particular, a topological disk.

On the other hand, let $\mathcal{W}_g$ be constituted of SW-surfaces of CGC-type and let $\Sigma \in \mathcal{W}_g$ satisfy the conditions of Theorem \ref{main-theorem}. In particular, $\Sigma$ has positive Gaussian curvature everywhere. Assume by contradiction  that $\Sigma$ is not a topological disk, then its Euler characteristic must be less or equal than $-1$. By applying the Gauss-Bonnet
Theorem to $\Sigma,$ we have that
\begin{equation*}
\int_{\Gamma} k_g(s) ds \leq -2 \pi - \int_{\Sigma}K dV_{\Sigma} < -2 \pi, 
\end{equation*}
where $k_g$ is the geodesic curvature function of $\Gamma$ with respect to $\Sigma$ and $s$ is the  arc-length parameter  of $\Gamma$. Now, we apply the Gauss-Bonnet Theorem to the planar domain $\Omega$  bounded by $\Gamma$ and we obtain
\begin{equation*}
\left\vert \int_{\Gamma} k_g(s)ds \right\vert \leq \int_{\Gamma} \vert k_g(s) \vert ds \leq \int_{\Gamma} \vert k(s) \vert ds = \int_{\Gamma} k(s) ds = 2\pi,
\end{equation*}
where $k$ is the curvature of $\Gamma$ in $P$. This is a contradiction. Therefore $\Sigma$ is a topological disk.
\end{remark}

\begin{remark} \label{remarkcircle}
In the particular case where $\Gamma \subset \{z=0\}$ is a circle, Ros-Rosenberg result is also elementary, whatever $\mathcal{W}_g$ is of minimal-type, CMC-type or CGC-type. In this case, for any $f \in [g]$, if $\Sigma \in \mathcal{W}_f$ is any surface embedded in $\mathbb{R}^3_{+}$ whose boundary is $\Gamma$, the Alexandrov reflection method provides that $\Sigma$ is symmetric with respect to any vertical plane. Thus $\Sigma$ must be a surface of revolution whose axis is the vertical line passing through the center of $\Gamma$, and by the classification of SW-surfaces of revolution in \cite{FeMi} and the fact that $\Sigma$ is compact, such a revolution surface must be a spherical cap.
\end{remark}

\begin{remark}
Our Theorem \ref{main-theorem} implies \cite[Theorem 2]{RoRo}. Indeed, let us take $g(t)=2-t$, which trivially satisfies condition \eqref{assumption1} for $m_0 = -1$. The corresponding set $\mathcal{W}_g$ consists of surfaces with constant mean curvature $H=1$. Then, if $d_0(\Gamma)>0$ is the number provided in the statement of Theorem \ref{main-theorem}, by taking $h(\Gamma) := 1/d_0(\Gamma)$ the conclusion follows for any $f \in [g]$ given by $f(t) = g(dt)/d$ with $d > d_0(\Gamma)$. The corresponding set $\mathcal{W}_f$ consists of surfaces with constant mean curvature $H < h(\Gamma)$.
\end{remark}

\begin{remark}
The hypothesis \eqref{assumption1} of Theorem \ref{main-theorem} includes classic sets of SW-surfaces: for example, constant mean curvature surfaces, any linear Weingarten relation of the form $\kappa_2 = c \kappa_1 + (1-c)\alpha$, for any $\alpha > 0$ and $c < 0$, and more generally any set of uniformly elliptic Weingarten surfaces $\mathcal{W}_g$, that is, when there exist constants $c_0 < C_0 < 0$ suth that the function $g:[\alpha,+\infty) \rightarrow (-\infty,\alpha]$ satisfies $c_0 \leq g'(t) \leq C_0$, for all $t \in [\alpha,+\infty)$.
\end{remark}

\begin{proof}[Proof of Theorem \ref{main-theorem}]
In view of Remark \ref{remarkcircle}, we may suppose that $\Gamma$ is not a circle, that is, $\Lambda > \lambda > 0$.

Recall that, for any $f \in [g]$ there is $d > 0$ such that $f(t) = g(dt)/d$, for $t \in [\alpha/d,+\infty)$. Then the constant of umbilicity and the positive principal curvature of the cylinders of $\mathcal{W}_f$ with respect to the inner unit normal are given by $\alpha/d$ and $\beta/d$, respectively. Our objective here is to infer the conditions on $d$ to obtain a lower bound $d_0$, such that  the conclusion of the statement holds for $\mathcal{W}_f$, if $f(t) = g(dt)/d$ and $d>d_0$. The following construction is based on arguments found in \cite{NePi} and \cite{NePiRu}.

Let $\Sigma \in \mathcal{W}_f$ be a compact surface embedded in $\mathbb{R}^3_+$ whose boundary spans $\Gamma$ and let $h_\Sigma$ be the maximum height of $\Sigma$ with respect to the horizontal plane $P$. We will divide the proof into two cases according to whether $h_\Sigma < 2d/\beta$ or $h_\Sigma \geq 2d/\beta$.

We recall that $D(p,r)$ denotes the horizontal closed disk of radius $r$ and center $p$ and if $I \subset \mathbb{R}$ is any interval, we denote by $D(p,r) \times I$ the intersection of the solid cylinder on  $D(p,r)$ with the closed slab $\{(x,y,z) \in \mathbb{R}^3; z \in I\}$.

\vspace{.3cm}
\textbf{CASE 1:} $h_\Sigma < 2d/\beta$.
\vspace{.3cm}

We will show that $\Sigma$ is a vertical graph over $\Omega$, implying in particular that it is topologically a closed disk. 

Let us fix a generic unit horizontal vector $v$ and let $Z$ be the horizontal cylinder in $\mathcal{W}_f$ whose revolution axis is the straight line orthogonal to $v$, contained in the plane $\{z = d/\beta\}$ and passing through $(c_\Gamma, d/\beta)$. In this way, the plane $P$ is tangent to $Z$, and the maximum height of $Z$ with respect to $P$ is equals to $2d/\beta$. If $Q_v$ is the vertical plane orthogonal to $v$ passing through $(c_\Gamma, d/\beta)$, we define $\widetilde{Z}$ as the part of $Z$ lying below $Q_v$ with respect to $v$, and $\widetilde{Z}(t) := \widetilde{Z} - t v$ as the image of $\widetilde{Z}$ by the translation with respect to the vector $-t v$.

For values $t \gg 0$, $\widetilde{Z}(t)$ does not touch $\Sigma$. By decreasing $t$, $\widetilde{Z}(t)$ approaches $\Sigma$ and by the tangency principle, there cannot be a first contact point between $\widetilde{Z}(t)$ and $\Sigma$ for $t > \omega(\Gamma)$. This implies that $\Sigma$ is contained in the upper (with respect to $v$) closed half-space bounded by $Q_v - (\omega(\Gamma) + d/\beta) v$. As $v$ was generically chosen, we deduce that $\Sigma$ is contained in $D((c_\Gamma,0),\omega(\Gamma) + d/\beta) \times [0,h_\Sigma]$.

Let $S \in \mathcal{W}_f$ be the round sphere of radius $d/\alpha$ centered in $(c_\Gamma,0)$ and let $S^{+}$  be the closed hemisphere of $S$ defined by   $S^{+} = S \cap \{0 \leq z \leq d/\alpha\}$. Let us choose $d > 0$ (depending only on $g$ and $\lambda$) such that 

\begin{equation} \label{conditioncase1}
1/\lambda + d/\beta < d/\alpha.
\end{equation}

This yields that $\omega(\Gamma) + d/\beta \leq 1/\lambda + d/\beta < d/\alpha$ and then, the cylinder $\partial S^{+}\times [0,+\infty)$ strictly contains the truncated solid cylinder $D((c_\Gamma,0),\omega(\Gamma) + d/\beta) \times [0,h_\Sigma]$ in its interior. Define $S(\tau) := S^{+} + \tau e_3$ to be the translation of $S^{+}$ by the vector $\tau e_3$. For any $\tau > 2d/\beta,$ the surfaces $S(\tau)$ and $\Sigma$ do not meet. By decreasing $\tau$, if there is a first contact point between $S(\tau_0)$ and $\Sigma$, for some $\tau_0 \geq 0$, such a contact point must be necessarily an interior tangency point, since the domain bounded by the orthogonal projection of the boundary of $S(\tau_0)$ on $P$ strictly contains $\Gamma$. The existence of such an interior contact point is a contradiction by the tangency principle. Therefore, $\Sigma$ must be strictly below $S(0)$.

Since $1/\lambda < d/\alpha$, we can translate $S(0)$ horizontally in such a way that its boundary touch every point of $\Gamma$ we desire from outside. Again by the tangency principle, this implies that $\Sigma$ is contained in the solid cylinder $\overline{\Omega}\times[0,\infty)$. Now, it is easy to prove, using the Alexandrov reflection method  with  horizontal planes, that $\Sigma$ is a graph over $\Omega$ and therefore a topological disk.

\vspace{.3cm}
Notice that up to now we did not need assumption \ref{assumption1} on $g$. As a matter of  fact, we only use assumption \ref{assumption1} in Case 2.

\vspace{.3cm}
\textbf{CASE 2:} $h_\Sigma \geq 2d/\beta$.
\vspace{.3cm}

Let $G$ be the SW-catenoid defined in Section \ref{section-catenoid} where $m_0$ is defined in \eqref{assumption1} and whose neck is a circle $C_0$ in $P$ of radius $r_0 = r(\Gamma)$ given in Lemma \ref{Lemasmalldisk}. Consider $G^\ast$ as the part of $G$ between the neck $C_0$ and the horizontal section $C_1$ at height $h^\ast := r(\Gamma) \mathfrak{h}^\ast(m_0)$, where $\mathfrak{h}^\ast(m_0)$ is defined in \eqref{h-ast}. Finally, denote by $r_1$ the radius of $C_1$. Notice that $h^\ast$ and $r_1$ only depends on $\Gamma$ and $[g]$.
 
Let $V$ be the set of the horizontal vectors $v$ for which $C_0 + v \subset \Omega$. By taking $R_1:= \omega(\Gamma) + r_1 - r_0 > 0$, for all $v \in V$ we have $C_1 + v \subset D((c_\Gamma,h^\ast),R_1)$. We observe that $R_1$ depend only on $\Gamma$, $[g]$ and $h^\ast$ (see Fig. \ref{initialconsiderations}, left).
\begin{figure}[h]
\centering
\includegraphics[scale=.9]{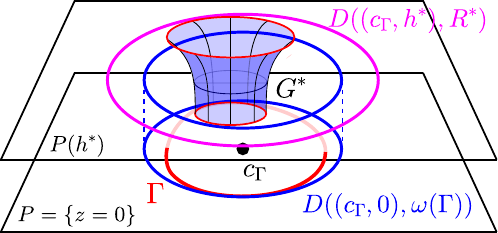}
\hspace{.5cm}
\includegraphics[scale=.6]{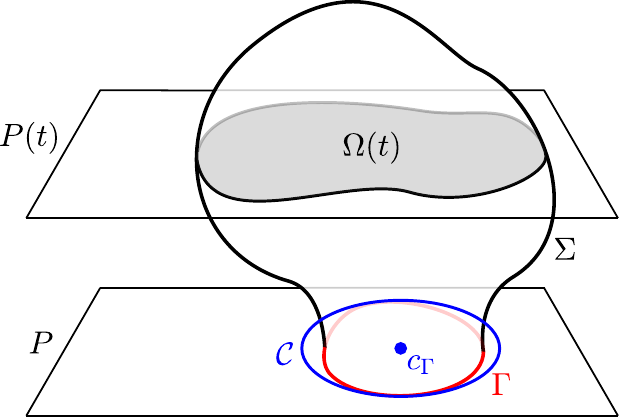}
\caption{on the left, the disk $D((c_\Gamma,h^\ast),R_1)$; on the right, the open region $\Omega(t)$.}
\label{initialconsiderations}
\end{figure}

For any $t \geq 0$, we define $\Omega(t)$ as the open region of the horizontal plane $P(t):= \{z = t\}$ bounded by $\Sigma \cap P(t)$ (see Fig. \ref{initialconsiderations}, right). For almost all values $t \geq 0$, $\Omega(t)$ is well defined. For later use, we need to define the following number:
\begin{equation} \label{defrf}
r_d := \sqrt{\frac{2 d h^\ast}{\alpha} - (h^\ast)^2}.
\end{equation}
The condition required to $r_d$ be well defined is that
\begin{equation} \label{condition0}
d > \frac{1}{2}\alpha h^\ast.
\end{equation}
Moreover, in order to  guarantee the next construction makes sense, we need that 

\begin{equation} \label{condition0a}
h^\ast < h_\Sigma/2.
\end{equation}

As we are working in CASE 2, that is, with the assumption $h_\Sigma \geq 2 d/\beta$, condition \eqref{condition0a} is guaranteed by requiring that 
 
\begin{equation} \label{condition0b}
d > \beta h^\ast.
\end{equation}

Notice that condition \eqref{condition0b} implies condition \eqref{condition0} as $\beta > \alpha$.

\vspace{.3cm}
\textbf{Claim 1.} \emph{The closure of $\Omega(h_\Sigma - h^\ast)$ contains two points $p$ and $q$ such that $\mbox{\rm{dist}}(p,q) \geq r_d$.}
\vspace{.3cm}

\begin{proof}[Proof of Claim 1]
Let $p_{\max} \in \Sigma$ be a point at maximal height and let $\mathcal{O}$ denote the open region bounded by $\Sigma \cup \Omega$. By doing the Alexandrov reflection method with horizontal planes, we get that the reflection of $p_{\max}$ with respect any plane $P(t)$ for $h_\Sigma/2 < t < h_\Sigma$ lies in $\mathcal{O}$. 
This means the reflection of $p_{\max}$ with respect to $P(h_\Sigma/2)$ lies in the closure of $\Omega$ and the orthogonal projection $p$ of $p_{\max}$ on $P(h_\Sigma - h^\ast)$ lies in the closure of $\Omega(h_\Sigma - h^\ast)$. Up to horizontal translations, we can assume that $p_{\max}$ lies in the axis $\{(0,0)\} \times \mathbb{R}$. We shall show that there exists at least a point $q$ in the closure of $\Omega(h_\Sigma - h^\ast)$ whose distance from $p$ is greater or equal than $r_d$. Suppose that this does not occur and denote by  $\Sigma^\ast$ the connected component of $\Sigma \cap \{h_\Sigma - h^\ast \leq z \leq h_\Sigma\}$ containing $p_{\max}$. Then all the points of $\partial\Sigma^\ast = \Sigma^\ast \cap P(h_\Sigma - h^\ast)$ are at a distance less than $r_d$ from the axis $\{(0,0)\} \times \mathbb{R}$. Consider the round sphere of radius $d/\alpha$ centered at $(0,0,h^\ast - d/\alpha)$  in  $\mathcal{W}_f$, and define $S^{+}$ as its connected component  in the upper closed half-space $\{z \geq 0\}$. Then $S^{+}$ is a spherical cap of height $h^\ast$ whose boundary is the circle in $P$, centered at the origin with radius $r_d$. Let $S^{+}(t)$ be the vertical translation of $S^{+}$ by the vector $t e_3$. For each $t > h_\Sigma$, $S^{+}(t)$ does not touch $\Sigma^\ast$. Then, we decrease $t$ and from our assumption, we have that for any $t \in [h_\Sigma - h^\ast,h_\Sigma]$ the boundary of $S^{+}(t)$ does not meet the boundary of $\Sigma^\ast$. Moreover, for all $t > h_\Sigma - h^\ast$ there is no interior contact point between $S^{+}(t)$ and $\Sigma^\ast$, otherwise, by the tangency principle, $\Sigma$ would coincide with $S^{+}(t)$ and hence it  would have a point at height larger than $h_\Sigma$. On the other hand, since $p_{\max}$ has height $h_\Sigma$, there should be a first interior contact point between $S^{+}(h_\Sigma - h^\ast)$ and $\Sigma^\ast$. By the tangency principle, $S^{+}(h_\Sigma - h^\ast)$ and $\Sigma^\ast$ must coincide. This is again a contradiction because the points at the boundary of $S^{+}(h_\Sigma - h^\ast)$ are at a distance $r_d$ from the axis $\{(0,0)\} \times \mathbb{R}$ while $\Sigma^\ast$ has no such points. 
\end{proof}

Before stating the next Claim, let us disclose in advance that we will do a lima\c con of Pascal construction (see Section \ref{limacon-section}) and we will require that the lima\c con has two loops. This will be accomplished if the following inequality holds:
\begin{equation}
\label{condition1a}
r_d > 2/\lambda,
\end{equation}
which, by \eqref{defrf}, also writes as follows
\begin{equation}\label{condition1}
\alpha/d < \frac{2 h^\ast}{4/{\lambda^2} + (h^\ast)^2}.
\end{equation}
By a straightforward computation, we get that  inequality \eqref{condition1} is verified if we choose $d > 0$ such that
\begin{equation}\label{condition1b}
d > \frac{\alpha}{2 h^\ast}(4/\lambda^2 + (h^\ast)^2)
\end{equation}
\vspace{.3cm}
\textbf{Claim 2.} {\em The open region $\Omega(h_\Sigma - h^\ast)$ contain a closed disk of radius
\begin{equation}
\label{ineq-claim2}
R^\ast \geq \sqrt{\frac{(r_d - 2 \omega(\Gamma))^3}{r_d - \omega(\Gamma)}}.
\end{equation}
Moreover the distance from its center $p^\ast$ to the straight line $\{c_\Gamma\} \times \mathbb{R}$ is $\omega(\Gamma)$.}
\vspace{.3cm}

\begin{proof}[Proof of Claim 2] 
Let $p$ and $q$ the points given in Claim 1. Recall that the orthogonal projection of $p$ over $P$ is contained in the closure of $\Omega$. We define $\mathcal{C}^\ast := \mathcal{C} \times \{h_\Sigma - h^\ast\},$ that is the circle of the plane $P(h_\Sigma - h^\ast)$ centered at $c_\Gamma^\ast := (c_\Gamma,h_\Sigma - h^\ast)$ with radius $\omega(\Gamma)$. In particular, as condition \eqref{condition1a} guarantees that  $r_d  > 2\omega(\Gamma)$ and  as $p$ is contained in the closed domain bounded by $\mathcal{C}^\ast,$ we deduce that $q$ lies outside the compact planar domain bounded by $\mathcal{C}^\ast$ in $P(h_\Sigma - h^\ast)$. Let $Q$ be any vertical plane tangent to $\mathcal{C}^\ast$ and consider $q(Q)$ the reflection of $q$ across the plane $Q$. The geometric locus of the points $q(Q)$ is a lima\c con of Pascal $\mathcal{L}$ of $P(h_\Sigma - h^\ast)$ with respect to the circle $\mathcal{C}^\ast$ and base point $q$. We denote by $a = \mbox{dist}(c_{\Gamma}^\ast, q)$ and $c = \omega(\Gamma)$ the parameters of $\mathcal{L}$. Using the triangular inequality, we have that
\begin{equation}\label{rh1}
r_d = \mbox{dist}(p,q) \leq \mbox{dist}(p,c_\Gamma^\ast) + \mbox{dist}(q,c_\Gamma^\ast) \leq \omega(\Gamma) + \mbox{dist}(q,c_\Gamma^\ast) = \omega(\Gamma) + a.
\end{equation}
and since the second parameter $c$ is the radius $\omega(\Gamma)$ of $\mathcal{C}^\ast$ we deduce that
\begin{equation} \label{l1l2}
a - c \geq r_d - 2 \omega(\Gamma) > 0.
\end{equation}

Inequality \eqref{l1l2} guarantees that $\mathcal{L}$ has two loops (see Fig. \ref{limaconClaim2}). Using similar arguments as in the proof of Lemma \ref{Lemasmalldisk}, we deduce that the smaller loop $\mathcal{L}'$ of $\mathcal{L}$ bounds a closed domain contained in $\Omega(h_\Sigma - h^\ast)$.
\begin{figure}[h]
\centering
\includegraphics[scale=.7]{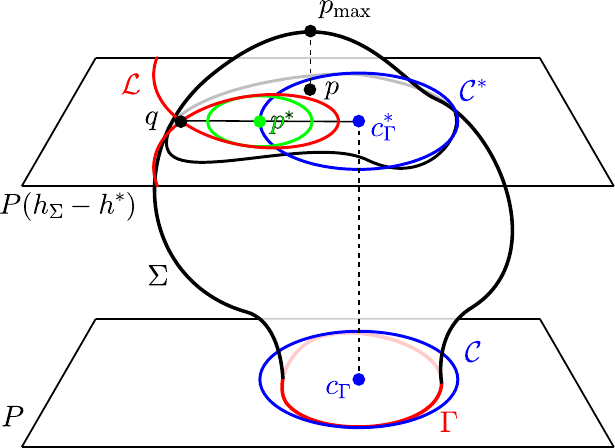}
\caption{the construction of the lima\c con of Pascal in Claim 2.}
\label{limaconClaim2}
\end{figure}

If we consider $p^\ast \in \mathcal{C}^\ast$ as the point of least distance from $q$, by the properties of the lima\c con of Pascal, $\mathcal{L}'$ bounds a closed domain which contains $p^\ast$ in its interior. To finish the proof, we take the greatest closed disk centered at $p^\ast$ that is contained in the closed domain bounded by $\mathcal{L}'$. Denoting by $R^\ast$ the radius of this disk, by (1) of Lemma \ref{lemma2_lima}, $R^\ast = \sqrt{\frac{(a - c)^3}{a}}$. From inequality \eqref{rh1} and Remark \ref{rmk_lima}, the quantity one has

\begin{equation*}
R^\ast = \sqrt{\frac{(a - c)^3}{a}} \geq \sqrt{\frac{(r_d - 2 \omega(\Gamma))^3}{r_d - \omega(\Gamma)}},
\end{equation*}
finishing the proof of Claim 2.
\end{proof}

\vspace{.3cm}
Let $c^\ast$ be the orthogonal projection of $p^\ast$ on the plane $P$. A sufficient condition implying that the disk $D(c_\Gamma,R_1)$ described in the beginning of the proof of CASE 2 is contained in the disk $D(c^\ast,R^\ast)$ is that
\begin{equation}\label{sufficientinequality}
R^\ast > \mbox{dist}(c^\ast,c_\Gamma) + R_1 = 2\omega(\Gamma) + r_1 -r_0.
\end{equation}
In view of inequality \eqref{ineq-claim2}, to obtain inequality \eqref{sufficientinequality} is sufficient to impose that 
\begin{equation} \label{condition2}
\sqrt{\frac{(r_d - 2 \omega(\Gamma))^3}{r_d - \omega(\Gamma)}} > 2\omega(\Gamma) + r_1 - r_0.
\end{equation}
We now look for which condition we must infer on $d>0$ in order to condition \eqref{condition2} be verified. Note that the function $j:[2\omega(\Gamma),+\infty) \rightarrow \mathbb{R}$ given by
\begin{equation*}
j(s) = \sqrt{\frac{(s - 2 \omega(\Gamma))^3}{s - \omega(\Gamma)}},
\end{equation*}
is strictly increasing and $\lim_{s \rightarrow +\infty}j(s) = + \infty$, regardless the value of $\omega(\Gamma) > 0$. Then, there exists $ s_0(\Gamma) \in [2\omega(\Gamma),+\infty)$ such that $j(s) > 2\omega(\Gamma) + r_1 - r_0$ for any $s \geq s_0(\Gamma)$. By the definition of $r_d$ (see \eqref{defrf}), in order to have $r_d \geq s_0(\Gamma)$, it is enough to choose $d > 0$ such that
\begin{equation} \label{condition2b}
d \geq \frac{\alpha}{2 h^\ast}(s_0(\Gamma)^2 + (h^\ast)^2).
\end{equation}
Here again the condition on $d$ depend only on $\Gamma$ and $[g]$.

\vspace{.3cm}
\textbf{Claim 3:} The intersection between $\Sigma$ and the truncated solid cylinder $D(c^\ast,R^\ast) \times [h^\ast,h_\Sigma - h^\ast]$ is empty.
\vspace{.3cm}

\begin{proof}[Proof of Claim 3] 
From Claim 2, we know that $D(p^\ast,R^\ast)$ is contained in $\Omega(h_\Sigma - h^\ast) \subset \mathcal{O}$. Denote by $\Sigma(t)$ the part of $\Sigma$ above $P(t)$ and by $\widetilde{\Sigma}(t)$ the reflection of $\Sigma(t)$ with respect to $P(t)$. The Alexandrov reflection method  with horizontal planes implies that $\Sigma(t)$ is a vertical graph and $\widetilde{\Sigma}(t)$ is contained in $\mathcal{O}$, for all $t \in [\frac{h_\Sigma}{2},h_\Sigma]$. From these facts we conclude that for all $t \in [\frac{h_\Sigma}{2},h_\Sigma - h^\ast]$ the reflection of $D(p^\ast,R^\ast)$ with respect to $P(t)$ is contained in $\mathcal{O}$, whence we deduce that $D(c^\ast,R^\ast) \times [h^\ast,h_\Sigma - h^\ast]$ is also contained in $\mathcal{O}$.
\end{proof}

\vspace{.3cm}
\textbf{Claim 4:} The intersection between $\Sigma$ and the region given by $\Omega \times [0,h^\ast]$ is empty.
\vspace{.3cm}

\begin{proof}[Proof of Claim 4]
To prove Claim 4, we consider the truncated catenoid $G^\ast$ with boundary $C_0 \cup C_1$ defined in the beginning of the proof of CASE 2. We need that a suitable translation of $G^\ast$ fits in the compact region $D(c^\ast,R^\ast) \times [h^\ast,h_\Sigma - h^\ast]$. Then, one needs $(h_\Sigma - h^\ast) - h^\ast \geq h^\ast$. Using the general assumption of CASE 2 that $h_\Sigma \geq 2 d/\beta$, the previous condition is accomplished by requiring that
\begin{equation} \label{condition4}
d > \frac{3}{2}\beta h^\ast.
\end{equation}
Notice that, inequality \eqref{condition4} implies inequality \eqref{condition0b}.

Choosing $d > 0$ satisfying conditions \eqref{condition1b}, \eqref{condition2b} and \eqref{condition4}, we have that $R^\ast$ satisfies inequality \eqref{sufficientinequality}, and by Claim 3 we can consider a translated copy $\widetilde{G}$ of $G^\ast$ that is strictly contained in $D(c^\ast,R^\ast) \times [h^\ast,h_\Sigma - h^\ast]$. Moreover, if $c_0 \in P$ and $r_0 = r(\Gamma)$ denote the center and radius of the disk given by Lemma \ref{Lemasmalldisk}, respectively, we can suppose additionally that the smaller boundary of $\widetilde{G}$ coincides with the boundary of $D((c_0,h^\ast),r_0)$.

Now we denote by $\widetilde{G}(t)$ the image of $\widetilde{G}$ by the translation through the vector $-t e_3$. There is no values of $t \in [0,h^\ast]$ for which the boundary of $\widetilde{G}(t)$ meets $\Sigma$, because the orthogonal projection of the smaller boundary of $\widetilde{G}(t)$ over $P$ is contained in $D(c_0,r_0) \subset \Omega$, and the larger boundary of $\widetilde{G}(t)$ is contained in $D(c^\ast,R^\ast) \times [h^\ast,h_\Sigma - h^\ast]$, which in turn is contained in $\mathcal{O}$. Thus, a possible first contact point between $\widetilde{G}(t_1)$ and $\Sigma$ for some $t_1 \in [0,h^\ast]$ would be an interior point for both surfaces. If this happens, one has a contradiction by the tangency principle. Let us prove it. Since $G$ is a special catenoid given by $\kappa_2 = m_0 k_1$, it is inward oriented and since assumption \eqref{assumption1} implies that $f(t) \geq (1-m_0)\alpha + m_0 t$, for all $t \in [\alpha/d, +\infty)$, the curvature diagram of $G$ lies strictly below the curvature diagram corresponding to the set $\mathcal{W}_f$ (see Fig. \ref{figureDCPhypothesis}, left). Then we can apply the tangency principle at a first interior contact point of $\widetilde{G}(t_1)$ and $\Sigma$ to conclude that they must coincide, which is a contradiction. Therefore, there is no interior contact point for all $t \in [0,h^\ast]$ and $\widetilde{G}(h^\ast)$ is contained in $\mathcal{O}$. Notice that, the smaller boundary of $\widetilde{G}(h^\ast)$ is contained in $\Omega$. Moreover, since $r_0=r(\Gamma) \leq 1/\Lambda$, then, for any point of $w \in \Gamma$, there is a horizontal translation of $\widetilde{G}(h^\ast)$ such that its smaller boundary is tangent $\Gamma$ at $w$. Notice that, because of inequality \eqref{sufficientinequality}, by performing such translations, the upper boundary of the translations of $\widetilde{G}(h^\ast)$ stays inside $\mathcal{O}$. Finally, there cannot be any interior contact point between the horizontal translations of $\widetilde{G}(h^\ast)$ and $\Sigma$ by the tangency principle. Then the intersection between $\Sigma$ and the region given by $\Omega \times [0,h^\ast]$ is empty.
\end{proof}

Now, we choose $d$ satisfying the conditions \eqref{conditioncase1}, \eqref{condition1b}, \eqref{condition2b} and \eqref{condition4} and we conclude the proof  by the following arguments.
\begin{itemize}
\item[(1)] By Claim 4, $\Sigma$ is outside the region $\Omega \times [0,h^\ast]$.
\item[(2)] By construction $D(c^\ast,R^\ast)$ contains $\Omega$. This fact, together with Claim 3, gives that $\Sigma \cap \{h^\ast \leq z \leq h_\Sigma - h^\ast\}$ is outside the region $\overline{\Omega} \times [h^\ast,h_\Sigma - h^\ast]$.
\item[(3)] From (1), (2) and the fact that $h_\Sigma - h^\ast > \frac{h_\Sigma}{2}$, we deduce that $\Sigma_0 := \Sigma \cap \{0 \leq z \leq \frac{h_\Sigma}{2}\}$ is outside the region $\overline{\Omega} \times [0,\frac{h_\Sigma}{2}]$, and by the Alexandrov reflection method with vertical planes, $\Sigma_0$ is a graph with respect to the radial directions orthogonal to $\Gamma \times [0,\frac{h_\Sigma}{2}]$. In particular, $\Sigma_0$ has the topology of an annulus.
\item[(4)] The Alexandrov reflection method for horizontal planes implies that $\Sigma_1 := \Sigma \cap \{\frac{h_\Sigma}{2} \leq z \leq h_\Sigma\}$ is a vertical graph.
\item[(5)] Since $\Sigma = \Sigma_0 \cup \Sigma_1$ and $\Sigma_0 \cap \Sigma_1$ is a closed simple curve (because $\Sigma$ is embedded and $\Sigma_0$ is a radial graph with respect to the directions orthogonal to $\Gamma \times [0,\frac{h_\Sigma}{2}]$), we conclude that both $\Sigma_1$ and $\Sigma$ have the topology of a closed disk.
\end{itemize}
\end{proof}

Our next objective is to prove Ros-Rosenberg Theorem without assuming condition \eqref{assumption1} on the function $g$, but only a geometrical condition (depending on the curvature of the given boundary $\Gamma$) on the umbilicity constant and the mean curvature of the standard cylinder in $\mathcal{W}_g$. 

Before stating the theorem, let us explain the strategy of the proof. The idea is to use the same construction as in  the proof of Theorem \ref{main-theorem}, with different choices for some of the parameters. 

Recall that the construction in the proof of Theorem \ref{main-theorem} was based on the conditions \eqref{conditioncase1}, \eqref{condition1b}, \eqref{condition2b} and \eqref{condition4}, where the choice of $d > 0$ sufficiently large was enough to get the conclusion. Notice that, the choice of $d$ determines a homothety of the involved surfaces.

In view of Assumption \eqref{assumption1}, we were able to  compare the homothety of $\Sigma$ with $G$ without restrictions (see Fig. \ref{figureDCPhypothesis}, left), because  for any surface $\Sigma$ in $\mathcal{W}_g$ and for any homothety of ratio $d > 0$, the principal curvature diagram of the image of $\Sigma$ by the homothety lies above $\kappa_2 = m_0 \kappa_1$, that is, the curvature diagram of the set of minimal-type SW-surfaces containing the SW-catenoid $G$. Now, since we do not assume condition \eqref{assumption1}, we do not have a natural choice for $m_0$. Furthermore, we consider broader sets of SW-surfaces, including for example those whose principal curvature diagram $\kappa_2=g(\kappa_1)$ is asymptotic to a vertical line. In this case, whatever is $m_0 < 0$, the principal curvature diagram of a homothety of a surface $\Sigma\in{\mathcal W}_g$ may not lie above the line $k_2 = m_0 k_1$ (see Fig. \ref{figureDCPhypothesis}, right).

\begin{figure}[h]
\centering
\includegraphics[scale=.8]{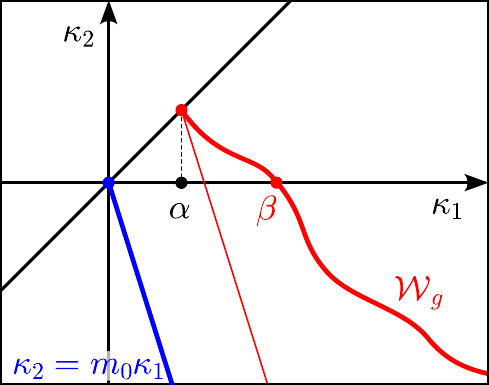}
\hspace{1.5cm}
\includegraphics[scale=.8]{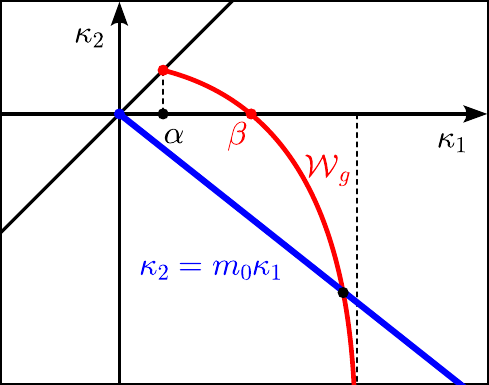}
\caption{on the left, the curvature diagram of $\mathcal{W}_g$ under assumption \eqref{assumption1}; on the right, an example of $\mathcal{W}_g$ whose curvature diagram is asymptotic to a vertical line.}
\label{figureDCPhypothesis}
\end{figure}

For the reasons above mentioned, we will avoid do homotheties, taking $d = 1$ in the proof of Theorem \ref{main-theorem} and write conditions \eqref{conditioncase1}, \eqref{condition1b}, \eqref{condition2b} and \eqref{condition4} in this case. Thus the choice of $m_0,$ as well the SW-catenoid $G$ for the relation $k_2 = m_0 k_1$ with neck-size $r(\Gamma)$ given in Lemma \ref{Lemasmalldisk}, and the truncated catenoid $G^\ast = G \cap \{0 \leq z \leq h^\ast\}$ will be attached to the following conditions.

\begin{itemize}
\item[(A)] $\dfrac{1}{\lambda} + \dfrac{1}{\beta} < \dfrac{1}{\alpha}$,
\item[(B)] $\alpha < \dfrac{2 h^\ast}{4/\lambda^2 + (h^\ast)^2}$,
\item[(C)] $\alpha < \dfrac{2 h^\ast}{4/s_0^2 + (h^\ast)^2}$,
\item[(D)] $\beta h^\ast < \dfrac{2}{3}$,
\end{itemize}
where $s_0 = s_0(m_0,\Gamma,h^\ast) \in [2\omega(\Gamma),+\infty)$ is the first value for which
\begin{equation}
j(s):= \sqrt{\dfrac{(s-2\omega(\Gamma))^3}{s-\omega(\Gamma)}} > 2\omega(\Gamma) + r_1 - r(\Gamma), \mbox{ for all } s \geq s_0,
\end{equation}
and $r_1 = r_1(m_0,\Gamma,h^\ast)$ is the radius of the horizontal circle $C_1:= G \cap \{z=h^\ast\}$, which can be computed through equation \eqref{heightparametrization}. Moreover, we need a further condition on $g$ that allows to compare $\Sigma\in{\mathcal W}_g$ with the SW-catenoid $G$ by the tangency principle. Such condition is:
\begin{itemize}
\item[(E)] $g(t) > m_0 t$, for all $t \in [\alpha, 1/r(\Gamma)]$;
\end{itemize}

Notice that, all the conditions (A)-(E)  concern the umbilicity constant $\alpha$ and the radius $\beta$  of the cylinder in ${\mathcal W}_g$. 
Moreover, conditions (A)-(E) are sufficient to guarantee that one can do the construction in the proof of Theorem \ref{main-theorem}. In other words, let $\Gamma \subset P:= \{z=0\}$ be a strictly convex planar curve with maximum and minimum curvature values $\Lambda > \lambda > 0$, and denote by $\omega(\Gamma)$ the radius of the smallest closed disk in $P$ containing $\Gamma$ and $r(\Gamma)$ the number given in Lemma \ref{Lemasmalldisk}. If one can find a number $m_0 < 0$ and $h^\ast < r(\Gamma)\mathfrak{h}(m_0)$ (see \eqref{eqbound}), then  for any  given   $g:[\alpha,b) \rightarrow \mathbb{R}$ such that the constant of umbilicity $\alpha$ and the curvature of its cylinders $\beta$ satisfy conditions (A)-(E),  any compact surface $\Sigma \in \mathcal{W}_g$ embedded in $\mathbb{R}^3_+$ whose boundary spans $\Gamma$ is topologically a disk.

Although the  conditions (A)-(E) represent  the most general  framework  based on the proof of Theorem \ref{main-theorem}, they are not easy to apply in practice.  Particular choices of $m_0$ and $h^\ast$ leads to the following, simpler to state but a bit more restrictive result for SW-surfaces of CMC-type.

\begin{theorem} \label{theorem2}
Let $\Gamma \subset P:= \{z = 0\}$ be a strictly convex, closed planar curve with maximum and minimum curvature values $\Lambda > \lambda > 0$ and let $\epsilon \in (0,2/3)$. Let $g:[\alpha,b) \rightarrow \mathbb{R}$, $b \in [\alpha, +\infty]$, be a $C^1$ function with $g(\alpha) = \alpha > 0$, $g(\beta) = 0$ for some $\beta \in (\alpha,b)$ and $g'(t) < 0$ for all $t \in [\alpha,b)$. Assume  that $\alpha$ and $\beta$ satisfy the following conditions
\begin{itemize}
\item[(1)] $\sqrt{\dfrac{\Lambda^3}{\lambda}} < \beta < \dfrac{2}{3\epsilon} \sqrt{\dfrac{\Lambda^3}{\lambda}};$
\item[(2)] $\alpha \leq C(\epsilon,\Lambda,\lambda):= \min \Biggl\{\dfrac{\Lambda\lambda}{\Lambda+\lambda}$, $\dfrac{8\epsilon \sqrt{\lambda^5\Lambda^3}}{\left( 9 \sqrt{\Lambda^3} + 2 \sqrt{\lambda^3}(\cosh\epsilon - 1)\right)^2 + 4\epsilon^2 \lambda^3}\Biggl\}$.
\end{itemize}
If $\Sigma$ is any compact surface embedded in $\mathbb{R}^3_+ = \{(x,y,z) \in \mathbb{R}^3; z \geq 0\}$ that belongs to $\mathcal{W}_g$ and whose boundary spans $\Gamma$, then $\Sigma$ is topologically a closed disk. Moreover, either $\Sigma$ is a vertical graph over the closed domain $\Omega$ bounded by $\Gamma$ in $P$, or $\Sigma \cap(\Omega \times [0,+\infty))$ is a graph over $\Omega$ and $\Sigma - (\Omega \times [0,+\infty))$ is a graph with respect to the lines normal to $\Gamma \times [0,+\infty)$.
\end{theorem}

\begin{example}
Let us consider $\Gamma$ the ellipse in the plane $\{z=0\}$ given by $x^2 + y^2/4 = 1$. Then, for $\epsilon = 1/2$, $\alpha = 1/25$ and $\beta = 1/10$,  the hypothesis of Theorem \ref{theorem2} are satisfied. If $g:[1/25,1/5) \rightarrow \mathbb{R}$ is given by
\begin{equation*}
g(t) = \dfrac{40t-4}{375t-75},
\end{equation*}
any compact surface $\Sigma$ of the corresponding set $\mathcal{W}_g$ embedded in $\mathbb{R}^3_+$ whose boundary is the ellipse $\Gamma$ must be topologically a disk. Notice also that $g$ does not satisfy condition \eqref{assumption1}, and the principal curvatures diagram of $\mathcal{W}_g$ is asymptotic to the vertical line $x = 1/5$.
\end{example}

\begin{remark}
Theorem \ref{theorem2} can be applied in the case of CMC surfaces with mean curvature satisfying a pinching condition, according to assumptions (1) and (2) with $\alpha = H$ and $\beta = 2H$.
\end{remark}

\begin{proof}[Proof of Theorem \ref{theorem2}]
A simpler set of conditions than (A)-(E) can be obtained by replacing condition E with
\begin{itemize}
\item[(E$^\prime$)] $\beta > \dfrac{1}{r(\Gamma)}$.
\end{itemize}
which is more restrictive, but permit us to choose any value $m_0 < 0$. More precisely, the principal curvature diagram of $\mathcal{W}_g$ lies above that of the SW-catenoid $G$ for the relation $k_2 = m_0 k_1$ of neck-size $r(\Gamma)$, whatever the choice of $m_0 < 0$.

Putting together conditions (D) and (E$^\prime$) we get
\begin{itemize}
\item[(F)] $\dfrac{3h^\ast}{2} < \dfrac{1}{\beta} < r(\Gamma)$.
\end{itemize}
and by taking $s_0 > \frac{2}{\lambda}$ we can merge conditions (B) and (C). Then we are left with conditions (A), (B) and (F).

As a consequence of condition (F), we can also replace condition (A) with
\begin{itemize}
\item[A$^\prime$.] $\dfrac{1}{\alpha} - \dfrac{1}{\lambda} > \dfrac{1}{\Lambda}$.
\end{itemize}
Indeed, if (A$^\prime$), (B) and (F) hold, then by Lemma \ref{Lemasmalldisk}
\begin{equation*}
\dfrac{1}{\lambda} + \dfrac{1}{\beta} < \dfrac{1}{\lambda} + r(\Gamma) < \dfrac{1}{\lambda} + \dfrac{1}{\Lambda} < \dfrac{1}{\alpha}.
\end{equation*}
Then we merge conditions (A$^\prime$) and (B) to get
\begin{itemize}
\item[(G)] $\dfrac{1}{\alpha} > \max\{\dfrac{1}{\lambda} + \dfrac{1}{\Lambda}, \dfrac{s_0^2 + (h^\ast)^2}{2 h^\ast}\}$.
\end{itemize}
We are left with conditions (F) and (G), but it remains to choose $m_0 < 0$ and $h^\ast$. We make the following particular but natural choices. First we define $r(\Gamma):= \sqrt{\frac{\lambda}{\Lambda^3}}$, which is the minimum value allowed by Lemma \ref{Lemasmalldisk}. Then we take $m_0 = -1$ (which corresponds to classic minimal surfaces) and $h^\ast = \epsilon r(\Gamma)$, for some $0< \epsilon < \frac{2}{3}$. Notice that the compatibility condition $h^\ast < r(\Gamma)\mathfrak{h}(-1)$ is trivially satisfied since $\mathfrak{h}(-1) = +\infty$. Then we have that $r_1 = r(\Gamma) \cosh \epsilon$. To compute $s_0$, we notice that
\begin{equation*}
j(s) = \sqrt{\dfrac{(s-2\omega(\Gamma))^3}{s-\omega(\Gamma)}} > s - \frac{5}{2}\omega(\Gamma), \quad \forall s > 2\omega(\Gamma).
\end{equation*}
If we make the choice
\begin{equation*}
s_0(\Gamma):= \dfrac{9}{2 \lambda} + r(\Gamma)(\cosh \epsilon - 1)
\end{equation*}
we have that $s_0 > 2/\lambda$ and for all $s \geq s_0$,
\begin{align*}
j(s) > s - \frac{5}{2}\omega(\Gamma) &> \dfrac{9}{2 \lambda} + r(\Gamma)(\cosh \epsilon - 1) - \frac{5}{2}\omega(\Gamma) \\
&> \dfrac{9}{2}\omega(\Gamma) + r(\Gamma)(\cosh \epsilon - 1) - \frac{5}{2}\omega(\Gamma) \\
&= 2 \omega(\Gamma) + r(\Gamma)(\cosh \epsilon - 1) \\
&= 2 \omega(\Gamma) + r_1 - r(\Gamma).
\end{align*}
From these choices of $r(\Gamma)$, $m_0$, $h^\ast$ and $s_0$, it is straightforward to check that conditions (F) and (G) translates into conditions (1) and (2) in the statement.
\end{proof}

{\bf  Acknowledgements}. The authors would like to thank Jos\'e Antonio G\'alvez for stimulating our interest in exploring Ros-Rosenberg theorem in the context of Weingarten surfaces, Isabel Fern\'andez for insightful suggestions and the anonymous referee for reading carefully the article and its valuable remarks. The authors were partially supported by PRIN-2022AP8HZ9, PRIN-20225J97H5 and INdAM-GNSAGA.


\begin{thebibliography}{abc99xys}


\bibitem[AlEsGa]{AlEsGa}
\textsc{J.A. Aledo, J.M. Espinar, J.A. G\'alvez}, \emph{The Codazzi equation for surfaces}, Adv. Math. \textbf{224} (2010) 2511--2530.


\bibitem[Al]{Al}
\textsc{A.D. Alexandrov}, \emph{Uniqueness theorem for surfaces in the large, III}, Vestn. Leningr. Univ. \textbf{13} (1958) 5--8; MR0102114,Zbl 0101.13902. English Translation: Amer. Math. Soc. Transl. \textbf{21} (1962) 412--416, MR0150710, Zbl 0119.16603.

\bibitem[BrSa]{BrSa}
\textsc{F. Brito, R. Sa Earp}, \emph{On the structure of certain Weingarten surfaces with boundary a circle}, Ann. de la faculté des sciences de Toulouse 6e série \textbf{6}(2) (1997) 243--255.

\bibitem[Br]{Br}
\textsc{R. Bryant}, \emph{Complex analysis and a class of Weingarten surfaces} (unpublished), arXiv:1105.5589.


\bibitem[BuOr]{BuOr}
\textsc{A. Bueno, I. Ortiz}, \emph{Surfaces of prescribed linear Weingarten curvature in $\mathbb{R}^3$}, Proc. Roy. Soc. Edinburgh Sect. A \textbf{153} (2023) 1347--1370.

\bibitem[CaCa]{CaCa}
\textsc{P. Carretero, I. Castro}, \emph{A new approach to rotational Weingarten surfaces}, Mathematics \textbf{10(4)} 578 (2022). 




\bibitem[Ch1]{Ch1}
\textsc{S.S. Chern}, \emph{On special W-surfaces}, Pr. Am. Math. Soc. \textbf{6} (1955) 783--786.

\bibitem[Ch2]{Ch2}
\textsc{S.S. Chern}, \emph{Some new characterizations of the euclidean sphere}, Duke Math. Journal \textbf{12} (1945) 279--290.


\bibitem[CoFeTe]{CoFeTe}
\textsc{A.V. Corro, W. Ferreira, K. Tenenblat}, \emph{Ribaucour transformations for constant mean curvature and linear Weingarten surfaces}, Pacific J. Math. \textbf{212} (2003) 265--297.


\bibitem[EsMe]{EsMe}
\textsc{J.M. Espinar, H. Mesa}, \emph{Elliptic special Weingarten surfaces of minimal type in $\mathbb{R}^3$ of finite total curvature}, 2019, preprint.
arXiv:1907.09122.

\bibitem[FeGaMi]{FeGaMi}
\textsc{I. Fernandez, J.A. G\'alvez, P. Mira}, {\em Quasiconformal Gauss maps and the Bernstein problem for Weingarten multigraphs},  Amer. J. Math. \textbf{145}(6) (2023) 1887--1921.



\bibitem[FeMi]{FeMi}
\textsc{I. Fernandez, P. Mira}, {\em Elliptic Weingarten surfaces: singularities, rotational examples and the halfspace theorem}, Nonlinear Analysis \textbf{232} (2023) 113244

\bibitem[GaMaMi]{GaMaMi}
\textsc{J.A. G\'alvez, A. Mart\'inez, F. Mil\'an}, \emph{Linear Weingarten surfaces in $\mathbb{R}^3$}, Monatsh. Math \textbf{138} (2003) 133--144.



\bibitem[GaMi1]{GaMi1}
\textsc{J.A. G\'alvez, P. Mira}, \emph{Rotational symmetry of Weingarten spheres in homogeneous three-manifolds}, J. Reine Angew. Math. \textbf{773} (2021) 21--66.

\bibitem[GaMi2]{GaMi2}
\textsc{J.A. G\'alvez, P. Mira}, \emph{Uniqueness of immersed spheres in three-manifolds}, J. Differential Geom. \textbf{116} (2020) 459--480.

\bibitem[GaMiTa]{GaMiTa}
\textsc{J.A. G\'alvez, P. Mira, M.P. Tassi}, \emph{A quasiconformal Hopf soap bubble theorem}, Calc. Var. Partial Differential Equations \textbf{61} (2022) 129.





\bibitem[HaWi]{HaWi}
\textsc{P. Hartman, A. Wintner}, \emph{Umbilical points and W-surfaces,} Amer. J. Math. \textbf{76} (1954) 502--508.

\bibitem[Ho]{Ho}
\textsc{H. Hopf}, \emph{Differential Geometry in the Large}, volume 1000 of Lecture Notes in Math. Springer-Verlag (1989).

\bibitem[Ko]{Ko}
\textsc{N. Koorevar}, \emph{Sphere theorems via Alexandrov for constant Weingarten curvature hypersurfaces - Appendix to a note of A. Ros}, J. Differential Geom. \textbf{27} (1988) 221--223.

\bibitem[KuSt]{KuSt}
\textsc{W. Kühnel, M. Steller}, \emph{On closed Weingarten surfaces}, Monatsh. Math. \textbf{146} (2005) 113--126.



\bibitem[LoPa]{LoPa}
\textsc{R. L\'opez, A. P\'ampano}, \emph{Classification of rotational surfaces in Euclidean space satisfying a linear relation between their principal curvatures}, Math. Nachr. \textbf{293} (2020) 735--753.


\bibitem[NePi]{NePi}
\textsc{B. Nelli, G. Pipoli}, \emph{Constant mean curvature hypersurfaces in $\mathbb{H}^n\times\mathbb{R}$ with small planar boundary}, Rev. Mat. Iberoam. \textbf{39}(4) (2023) 1387--1404.

\bibitem[NePiRu]{NePiRu}
\textsc{B. Nelli, G. Pipoli, G. Russo}, \emph{On constant higher order mean curvature hypersurfaces in $\mathbb{H}^n\times\mathbb{R}$}, Adv. Nonlinear Stud. 24 (1) (2024) 44--73.

\bibitem[Os]{Os}
\textsc{R. Osserman}, \emph{The convex hull property of immersed manifolds}, J. Differential Geom. \textbf{6} (1971) 267--270.

\bibitem[RoRo]{RoRo}
\textsc{A. Ros, H. Rosenberg}, \emph{Constant mean curvature surfaces in a half-space of $\mathbb{R}^3$ with boundary in the boundary of the halfspace}, J. Diff. Geom. \textbf{44} (1996) 807--817.

\bibitem[RoSa]{RoSa}
\textsc{H. Rosenberg, R. Sa Earp},  \emph{The geometry of properly embedded special surfaces in ${\mathbb R}^3$ e.g. surfaces satisfying $aH + bK = 1, $ where $a$ and $b$ are positive}, Duke Math. J. \textbf{73} (1994) 291--306.

\bibitem[SaTo1]{SaTo1}
\textsc{R. Sa Earp, E. Toubiana}, \emph{Sur les surfaces de Weingarten speciales de type minimal}, Bull. Braz. Math. Soc. \textbf{26} (1995) 129--148.

\bibitem[SaTo2]{SaTo2}
\textsc{R. Sa Earp, E. Toubiana}, \emph{Classification des surfaces de type Delaunay}, Amer. J. Math. \textbf{121} (1999) 671--700.

\bibitem[Se]{Se}
\textsc{B. Semmler}, \emph{The topology of large ${H}$-surfaces bounded by a convex curve}, Ann. scientifiques de l'Ecole Normale Sup. \textbf{33}(3) (2000) 345--359.
\end{thebibliography}
\end{document}